\documentclass[10pt,oneside]{amsart}

\usepackage[
paper=a4paper,
headsep=15pt,text={138mm,222mm},centering
]{geometry}

\usepackage[bookmarks]{hyperref}
\usepackage{amssymb,amsxtra,dsfont}
\usepackage[all]{xy}
\usepackage{pb-diagram,pb-xy}
\usepackage{graphicx,xcolor}
\usepackage{pinlabel}
\usepackage{float}
\usepackage{array,calc,booktabs}

\definecolor{todo-background-color}{gray}{0.95}
\usepackage[
  textwidth=1.1in,
  backgroundcolor=todo-background-color,
  bordercolor=black,
  linecolor=black,
  textsize=footnotesize
]{todonotes}

\iftrue
\makeatletter
\def\@settitle{%
  \vspace*{-0pt}
  \begin{flushleft}%
    \LARGE\bfseries
    \strut\@title\strut
  \end{flushleft}%
}
\def\@setauthors{%
  \begingroup
  \def\thanks{\protect\thanks@warning}%
  \trivlist
  \raggedright
  \large \@topsep27\p@\relax
  \advance\@topsep by -\baselineskip
  \item\relax
  \author@andify\authors
  \def\\{\protect\linebreak}%
  \authors
  \ifx\@empty\contribs
  \else
    ,\penalty-3 \space \@setcontribs
    \@closetoccontribs
  \fi
  \normalfont
  \endtrivlist
  \endgroup
}
\def\@setaddresses{\par
  \nobreak \begingroup
  \small\raggedright
  \def\author##1{\nobreak\addvspace\smallskipamount}%
  \def\\{\unskip, \ignorespaces}%
  \interlinepenalty\@M
  \def\address##1##2{\begingroup
    \par\addvspace\bigskipamount\noindent
    \@ifnotempty{##1}{(\ignorespaces##1\unskip) }%
    {\ignorespaces##2}\par\endgroup}%
  \def\curraddr##1##2{\begingroup
    \@ifnotempty{##2}{\nobreak\noindent\curraddrname
      \@ifnotempty{##1}{, \ignorespaces##1\unskip}\/:\space
      ##2\par}\endgroup}%
  \def\email##1##2{\begingroup
    \@ifnotempty{##2}{\nobreak\noindent E-mail address%
      \@ifnotempty{##1}{, \ignorespaces##1\unskip}\/:\space
      \ttfamily##2\par}\endgroup}%
  \def\urladdr##1##2{\begingroup
    \def~{\char`\~}%
    \@ifnotempty{##2}{\nobreak\noindent\urladdrname
      \@ifnotempty{##1}{, \ignorespaces##1\unskip}\/:\space
      \ttfamily##2\par}\endgroup}%
  \addresses
  \endgroup
  \global\let\addresses=\@empty
}
\def\@setabstracta{%
    \ifvoid\abstractbox
  \else
    \skip@17pt \advance\skip@-\lastskip
    \advance\skip@-\baselineskip \vskip\skip@
    \box\abstractbox
    \prevdepth\z@ 
    \vskip-28pt
  \fi
}
\renewenvironment{abstract}{%
  \ifx\maketitle\relax
    \ClassWarning{\@classname}{Abstract should precede
      \protect\maketitle\space in AMS document classes; reported}%
  \fi
  \global\setbox\abstractbox=\vtop \bgroup
    \normalfont\small
    \list{}{\labelwidth\z@
      \leftmargin0pc \rightmargin\leftmargin
      \listparindent\normalparindent \itemindent\z@
      \parsep\z@ \@plus\p@
      
    }%
    \item[\hskip\labelsep\bfseries\abstractname.]%
}{%
  \endlist\egroup
  \ifx\@setabstract\relax \@setabstracta \fi
}

\def\ps@headings{\ps@empty
  \def\@evenhead{%
    \setTrue{runhead}%
    \normalfont\scriptsize
    \rlap{\thepage}\hfill
    \def\thanks{\protect\thanks@warning}%
    \leftmark{}{}}%
  \def\@oddhead{%
    \setTrue{runhead}%
    \normalfont\scriptsize
    \def\thanks{\protect\thanks@warning}%
    \rightmark{}{}\hfill \llap{\thepage}}%
  \let\@mkboth\markboth
}\ps@headings

\def\section{\@startsection{section}{1}%
  \z@{-1.4\linespacing\@plus-.5\linespacing}{.8\linespacing}%
  {\normalfont\bfseries\Large}}
\def\subsection{\@startsection{subsection}{2}%
  \z@{-.8\linespacing\@plus-.3\linespacing}{.5\linespacing\@plus.2\linespacing}%
  {\normalfont\bfseries\large}}
\def\subsubsection{\@startsection{subsubsection}{3}%
  \z@{.7\linespacing\@plus.2\linespacing}{-1.5ex}%
  {\normalfont\bfseries}}
\def\paragraph{\@startsection{paragraph}{4}%
  \z@{.7\linespacing\@plus.2\linespacing}{-1.5ex}%
  {\normalfont\itshape}}
\def\@secnumfont{\bfseries}

\renewcommand\contentsnamefont{\bfseries}
\def\@starttoc#1#2{\begingroup
  \setTrue{#1}%
  \par\removelastskip\vskip\z@skip
  \@startsection{}\@M\z@{\linespacing\@plus\linespacing}%
    {.5\linespacing}{
      \contentsnamefont}{#2}%
  \ifx\contentsname#2%
  \else \addcontentsline{toc}{section}{#2}\fi
  \makeatletter
  \@input{\jobname.#1}%
  \if@filesw
    \@xp\newwrite\csname tf@#1\endcsname
    \immediate\@xp\openout\csname tf@#1\endcsname \jobname.#1\relax
  \fi
  \global\@nobreakfalse \endgroup
  \addvspace{32\p@\@plus14\p@}%
  \let\tableofcontents\relax
}
\def\contentsname{Contents}
\def\l@section{\@tocline{2}{.5ex}{0mm}{5pc}{}}
\def\l@subsection{\@tocline{2}{0pt}{2em}{5pc}{}}
\makeatother
\fi 

\def\to{\mathchoice{\longrightarrow}{\rightarrow}{\rightarrow}{\rightarrow}}
\makeatletter
\newcommand{\shortxra}[2][]{\ext@arrow 0359\rightarrowfill@{#1}{#2}}
\def\longrightarrowfill@{\arrowfill@\relbar\relbar\longrightarrow}
\newcommand{\longxra}[2][]{\ext@arrow 0359\longrightarrowfill@{#1}{#2}}
\renewcommand{\xrightarrow}[2][]{\mathchoice{\longxra[#1]{#2}}%
  {\shortxra[#1]{#2}}{\shortxra[#1]{#2}}{\shortxra[#1]{#2}}}

\def\addtagsub#1{\let\oldtf=\tagform@\def\tagform@##1{\oldtf{##1}\hbox{$_{#1}$}}}

\makeatother

\makeatletter
\def\Nopagebreak{\@nobreaktrue\nopagebreak}

\newtheoremstyle{theorem-giventitle}
        {}{}              
        {\itshape}                      
        {}                              
        {\bfseries}                     
        {.}                             
        {\thm@headsep}                             
        {\thmnote{\bfseries#3}}
\newtheoremstyle{theorem-givenlabel}
        {}{}              
        {\itshape}                      
        {}                              
        {\bfseries}                     
        {.}                             
        {\thm@headsep}                             
        {\thmname{#1}~\thmnumber{#3}\setcurrentlabel{#3}}

\newtheoremstyle{definition-giventitle}
        {}{}              
        {}                      
        {}                              
        {\bfseries}                     
        {.}                             
        {\thm@headsep}                             
        {\thmnote{\bfseries#3}}
\def\setcurrentlabel#1{\gdef\@currentlabel{#1}}

\makeatother

\newtheorem{theorem}{Theorem}[section]
\newtheorem{theoremalpha}{Theorem}
\newtheorem{corollaryalpha}[theoremalpha]{Corollary}
\newtheorem{proposition}[theorem]{Proposition}
\newtheorem{corollary}[theorem]{Corollary}
\newtheorem{lemma}[theorem]{Lemma}
\newtheorem{conjecture}[theorem]{Conjecture}

\theoremstyle{definition}
\newtheorem{definition}[theorem]{Definition}
\newtheorem{question}[theorem]{Question}

\newtheorem{remark}[theorem]{Remark}

\theoremstyle{theorem-giventitle}
\newtheorem{theorem-named}{}
\theoremstyle{theorem-givenlabel}
\newtheorem{theorem-labeled}{Theorem}

\theoremstyle{definition-giventitle}
\newtheorem{definition-named}{}
\newtheorem{step-named}{}

\numberwithin{equation}{section}

\def\Z{\mathbb Z}
\def\R{\mathbb R}
\def\Q{\mathbb Q}

\def\sm{\smallsetminus}

\DeclareMathOperator\Int{Int}

\DeclareMathOperator\Wh{Wh}

\def\la{\langle}
\def\ra{\rangle}
\def\id{\mathrm{Id}}

\def\ed{\text{--}}

\begin{document}

\title{Casson towers and slice links}

\author{Jae Choon Cha}
\address{
  Department of Mathematics\\
  POSTECH\\
  Pohang 790--784\\
  Republic of Korea\quad
  -- and --\linebreak
  School of Mathematics\\
  Korea Institute for Advanced Study \\
  Seoul 130--722\\
  Republic of Korea
}
\email{jccha@postech.ac.kr}

\author{Mark Powell}
\address{
  D\'{e}partment de Math\'{e}matiques\\
  UQAM \\
  Montr\'{e}al, QC\\
  Canada
}
\email{mark@cirget.ca}


\def\subjclassname{\textup{2010} Mathematics Subject Classification}
\expandafter\let\csname subjclassname@1991\endcsname=\subjclassname
\expandafter\let\csname subjclassname@2000\endcsname=\subjclassname

\begin{abstract}
  We prove that a Casson tower of height 4 contains a flat embedded
  disc bounded by the attaching circle, and we prove disc embedding
  results for height~2 and~3 Casson towers which are embedded into a
  4-manifold, with some additional fundamental group assumptions.  In
  the proofs we create a capped grope from a Casson tower and use a
  refined height raising argument to establish the existence of a
  symmetric grope which has two layers of caps, data which is
  sufficient for a topological disc to exist, with the desired
  boundary.  As applications, we present new slice knots and links by
  giving direct applications of the disc embedding theorem to produce
  slice discs, without first constructing a complementary
  4-manifold.  In particular we construct a family of slice knots
  which are potential counterexamples to the homotopy ribbon slice
  conjecture.
\end{abstract}

\maketitle

\section{Introduction}

This paper presents results on \emph{Casson towers} of height 2, 3 and
4 in dimension four, and applications to the problem of slicing knots
and links.

The disc embedding problem is one of the most important questions in
4-manifold topology.  Roughly speaking, when the disc embedding
problem can be solved, the surgery and $s$-cobordism programme for the
classification of 4-manifolds can be carried out as in the high
dimensional case.  In fact disc embedding in these contexts is
essentially equivalent to the Whitney trick, which is a key ingredient
for geometrically realising the algebraic cancellation of intersection
data.

M.~Freedman solved the disc embedding problem in simply connected
topological 4-manifolds, and as a consequence he was able to classify
such manifolds~\cite[Chapter~10]{Freedman-Quinn:1990-1} using surgery
theory.  Freedman's solution built upon the work of A.~Casson, who
introduced the influential idea of a Casson tower.  A Casson tower
arises as the trace of repeated attempts to eliminate intersections of
an immersed disc, the goal being to find a flat embedded
disc~\cite{Casson-1986-towers}.  Briefly speaking, the \emph{height}
of a Casson tower is the number of stages of iterated attempts.  A
Casson tower $T$, itself a 4-manifold, is endowed with a framed circle
$C=C(T)$ embedded in its boundary.  We ask whether there exists a flat
embedded disc with framed boundary~$C$.  See
Definition~\ref{definition:casson-tower} for details.

Casson considered a tower of infinite height, which is now called a
\emph{Casson handle}~\cite{Casson-1986-towers}.  He showed that a
Casson handle is \emph{proper homotopy equivalent} to an open
2-handle.  In the original proof of the celebrated disc embedding
theorem in dimension 4~\cite{Freedman:1982-1}, Freedman showed that a
Casson handle is \emph{homeomorphic} to an open 2-handle, and
consequently contains a flat embedded disc with framed
boundary~$C(T)$. A key ingredient of the proof was Freedman's
reimbedding theorem~\cite[Theorem~4.4]{Freedman:1982-1}, which says
that a height 6 Casson tower contains within it a height 7 tower
(see~\cite{Bizaca-casson-handles-algorithm} for a detailed
exposition).  Iterating this, it follows that a given height 6 tower
$T$ contains a Casson handle, and consequently contains a flat
embedded disc with framed boundary~$C(T)$.  Gompf and Singh improved
this disc embedding result by showing that height 5 Casson towers are
sufficient for reimbedding~\cite{gompf-singh-1984}.

From this a natural question arises: what is the minimal height of a
Casson tower required to obtain an embedded disc?

In Theorems~\ref{theorem:main}, \ref{theorem:main-height-3} and
\ref{theorem:main-height-2} below we give disc embedding results for
Casson towers of height 4, 3 and 2 respectively, under increasingly
strong assumptions on fundamental groups.  The height 2 result is
particularly useful for the study of knot and link concordance, since
it is often feasible to construct such a tower in $D^4$ bounded by a
knot or link.  

Work of Freedman in the 1980s and
90s~\cite{Freedman:1982-2,Freedman:1985-1,Freedman:1988-2,Freedman:1993-1},
and also later work such as that of
Freedman-Teichner~\cite{Freedman-Teichner:1995-2},
Friedl-Teichner~\cite{Friedl-Teichner:2005-1} and
Cochran-Friedl-Teichner~\cite{Cochran-Friedl-Teichner:2006-1},
produced slice knots and links of great interest.  Particular focus
was placed on the question of which Whitehead doubles are slice (see
Conjecture~\ref{conjecture:topological-whitehead-double} below), since
topological surgery problems in dimension four can be reduced to
atomic problems~\cite{Casson-Freedman:1984-1} which have solutions
precisely when such links are slice.

Using our height 2 Casson tower embedding theorem
(Theorem~\ref{theorem:main-height-2}), we extend the class of known
slice knots to include the new family of slice knots described in
Theorem~\ref{theorem:main-slice-knots}.  We apply the disc embedding
theorem to construct slice discs directly, rather than using the
topological surgery machine employed by many of the papers mentioned
above.  Our slice knots relate closely to the Topological Whitehead
Double Conjecture~\ref{conjecture:topological-whitehead-double}, give
potential counterexamples to the Homotopy Ribbon Slice
Conjecture~\ref{conjecture:homotopy-ribbon-slice}, and suggest a
possible connection between the Homotopy Ribbon Slice
Conjecture~\ref{conjecture:homotopy-ribbon-slice} and the
4-dimensional surgery conjecture.

\subsection{Casson towers of height four, three, and two}

We proceed to introduce our disc embedding results for Casson towers of
height four, three, and two.  Let $W$ be a 4-manifold with boundary.
A framed link $L \subset \partial W$ is \emph{slice} in $W$ if $L$
bounds a collection of disjointly embedded flat discs in~$W$, as
framed manifolds.

\subsubsection*{Height four}

Our first main result implies that a \emph{height 4} Casson tower is
in fact sufficient to obtain a flat embedded disc.  In fact we give a
stronger result.  Briefly, define a \emph{distorted Casson tower} by
introducing plumbings of the top stage discs into discs of stage two
or higher in a Casson tower (see
Definition~\ref{definition:distorted-tower}).

\begin{theoremalpha}
  \label{theorem:main}
  A distorted Casson tower $T$ of height~4 contains a topologically
  embedded flat disc bounded by $C(T)$ as a framed manifold.
\end{theoremalpha}

In other words, $C(T)$ is slice in $T$.  Since a Casson tower is
vacuously a distorted Casson tower, Theorem~\ref{theorem:main} holds
for an ordinary (non-distorted) Casson tower of height~4.  This
assertion seems to have been expected to be true by the experts, but
to the knowledge of the authors, no proof has appeared in the
literature; compare~\cite[Footnote~1]{Ray-2013-1}.

\subsubsection*{Height three}

It is not known in general whether a height 3 Casson tower $T$
contains an embedded disc with boundary~$C(T)$.  Progress has been
made by looking at special cases, as instigated
in~\cite{Casson-Freedman:1984-1}.  Freedman proved that the simplest
Casson tower of height~3, namely the tower with a single double point
at each stage, contains a disc~\cite{Freedman:1988-2}\footnote{In
  fact, Freedman showed that a two component link called
  ``Whitehead${}_3$'' bounds slicing discs in the 4-ball whose
  complement has free fundamental group.  This link is associated to
  the simplest Casson tower $T$ of height 3, as explained in our
  Section~\ref{subsection:kirby-diagram-for-distorted-casson-tower}.
  It turns out that each of the two slicing discs is ambiently isotopic
  to the standard disc in the 4-ball
  by~\cite[11.7A]{Freedman-Quinn:1990-1}.  It follows that the
  exterior of one slicing disc is $T$ and the other slicing disc is
  bounded by~$C(T)$.}.  We remark that completing the analogous
argument to our proof of Theorem~\ref{theorem:main} for a height~3
Casson towers would seem to require the surgery conjecture for
non-abelian free groups.  The corresponding statement to
Theorem~\ref{theorem:main} for general height~3 Casson towers would
therefore be rather interesting.  The main difficulty, as so often in
this subject, is to achieve $\pi_1$-nullity.

Instead of looking for null homotopies internally in Casson towers, we
can consider embedded Casson towers in a 4-manifold, and then try to
find null homotopies inside the 4-manifold.  For our height~3 result,
we use the notion of a \emph{good group}.  In this paper we use the
definition of~\cite{Freedman-Teichner:1995-1}, which defines a group
to be good if it satisfies the $\pi_1$-null disc lemma.  Note that
this differs from the definition of a good group
in~\cite{Freedman-Quinn:1990-1}.  For a precise description and
related discussion, see Definition~\ref{definition:good-group} and the
paragraph following it.
A result of Freedman-Teichner and Krushkal-Quinn
\cite{Freedman-Teichner:1995-1, Krushkal-Quinn:2000-1} tells us that a
group of subexponential growth is good, as is any group obtained from
good groups by extensions and direct limits.

A.\ Ray considered a framed grope bounded by $C(T)$ in a Casson
tower~$T$~\cite{Ray-2013-1}. Denote the first stage surface of this
grope by~$\Sigma(T)$.  This is an oriented surface embedded in $T$
with $\partial\Sigma(T)=C(T)$.  See Figure~\ref{figure:ray-surface}
and Proposition~\ref{proposition:grope-in-casson-tower} for more
details.  For a disjoint union of Casson towers $T=\bigsqcup T_i$,
denote $C(T):=\bigsqcup C(T_i)$ and $\Sigma(T) := \bigsqcup
\Sigma(T_i)$.  Denote a tubular neighbourhood of $C(T)$ in $\partial
T$ by~$\partial_-(T)$.

\begin{theoremalpha}
  \label{theorem:main-height-3}
  Let $W$ be a $4$-manifold with boundary and suppose that
  $T=\bigsqcup T_i$ is a collection of disjoint Casson towers $T_i$ of
  height~$3$ in $W$ such that $\partial_-(T)\subset \partial W$ and
  the image of $\pi_1(T_i \sm \Sigma(T_i)) \to \pi_1(W \sm\Sigma(T))$
  is a good group for each~$i$.  Then the framed link
  $C(T) \subset \partial W$ is slice in~$W$.
\end{theoremalpha}

\noindent We remark that this result concerns links and not just knots.

\subsubsection*{Height two}

For the height 2 case, we obtain a slicing result under a stronger
simple connectivity hypothesis.  In the statement $T_{p\ed q}$ denotes
the union of the stages $p$ through $q$ inclusive for a Casson
tower~$T$; see Definition~\ref{definition:stages-notation} for a more
precise description.

\begin{theoremalpha}
  \label{theorem:main-height-2}
  Let $W$ be a $4$-manifold with boundary and suppose $T$ is a Casson
  tower of height~$2$ embedded in $W$ such that the second stage
  $T_{2\ed 2}$ of $T$ lies in a codimension zero simply connected
  submanifold $V \subseteq \overline{W \sm T_{1\ed 1}}$.  Then the knot
  $C(T)\subset \partial W$ is slice in~$W$.
\end{theoremalpha}

In Theorems \ref{theorem:main-height-3} and
\ref{theorem:main-height-2}, the slice discs are contained in a
neighbourhood of a union of the tower itself and a collection of null
homotopies for double point loops constructed during the proofs.

\subsubsection*{Our proofs and gropes}

After Freedman's original proof of the disc embedding theorem using
Casson towers, the grope technology (see
Definitions~\ref{definition:grope}--\ref{definition:one-story-capped-tower})
has been developed in subsequent work by
Quinn~\cite{Quinn-1982-ends-of-maps},
Edwards~\cite{Edwards-annulus-conjecture},
Freedman-Quinn~\cite{Freedman-Quinn:1990-1},
Freedman-Teichner~\cite{Freedman-Teichner:1995-1},
Krushkal-Quinn~\cite{Krushkal-Quinn:2000-1} and others.  It turned out
that gropes are effective for proving the disc embedding theorem in
the non-simply connected case.

Gropes are in fact a key ingredient of our proofs.  For
height four and three, our arguments hinge on Ray's construction of a
framed grope inside a Casson tower~\cite{Ray-2013-1}.  It enables us
to connect the grope and Casson tower techniques.

In the grope setting the minimal data required for the existence of a
topological disc has been quite well optimised in the decades since
the original reference~\cite{Freedman-Quinn:1990-1} was written.  (The
optimisation has not been enough for the surgery conjecture to be
known, of course).  Up to date grope combinatorics were partially
written up in~\cite[Chapter~8]{teamfreedman} by the second author and
W.~Politarczyk as part of the lecture notes for Freedman's lectures
for the Max Planck Institute for Mathematics semester on 4-manifolds
in~2013.  In the hope that they represent a useful addition to the
literature, details relevant to the current paper which cannot be
found in the earlier literature (e.g.\ \cite{Freedman-Quinn:1990-1})
are given below: see Grope Height Raising
Lemma~\ref{lemma:grope-height-raising} and Cap Separation
Lemma~\ref{lemma:cap-separation}.  In fact the proof of the latter
lemma has not appeared anywhere before to the best of our knowledge.
The inclusion of these details is further justified by the following
corollary, which is proven by combining Ray's grope construction with
the Grope Height Raising Lemma~\ref{lemma:grope-height-raising}.

\begin{theorem-named}[Corollary~\ref{cor:grope-filtration}]
  A Casson tower $T$ of height~3 contains an embedded grope of
  height~$n$, with the same attaching circle $C(T)$ as the Casson
  tower, for all $n$.
\end{theorem-named}

This improves a result of Ray~\cite[Theorem~A~(i)]{Ray-2013-1}.
Further discussion of this corollary can be found in
Section~\ref{section:casson-tower-3-infinite-grope}.

The proof of the height 2 result, Theorem~\ref{theorem:main-height-2},
requires an entirely new construction, given in
Proposition~\ref{proposition:capped-grope-from-embedded-casson-tower},
of \emph{capped} gropes from Casson towers embedded in a 4-manifold,
under a certain fundamental group condition.  This depends on new
geometric arguments and some quite delicate combinatorics.  The
application to slice knots discussed next utilises this construction.

\subsection{Applications to slicing knots and links}

We apply our results on Casson towers to present new slice knots and
slice links in~$S^3$.  As usual, we say a knot or link in $S^3$ is
\emph{slice} if it is slice in~$D^4$.

\subsubsection*{New slice knots}

To state our results on knots, we recall that Milnor called a link $L$
in $S^3$ \emph{homotopically trivial} if its components admit disjoint
null-homotopies \cite{Milnor:1954-1}. That is, if there are maps
$h_i\colon D^2 \to S^3$ such that $L=\bigsqcup_i h_i(S^1)$ and
$h_i(D^2)\cap h_j(D^2)=\emptyset$ for $i\ne j$.  We also recall that a
band sum operation on a link $L$ is performed along an embedded band
$D^1 \times D^1$ which joins two components of~$L$ and such that
$\Int(D^1) \times D^1$ is disjoint from~$L$.  Denote the untwisted
Whitehead double of a link $L$ by~$\Wh(L)$.  The link $\Wh(L)$ is only
defined up to a sign choice for the clasp when doubling each component
of $L$.  Our theorems hold for any choices of signs.

\begin{theoremalpha}
  \label{theorem:main-slice-knots}
  Suppose that $L$ is an $m$-component homotopically trivial link, and $K$
  is a knot obtained from $\Wh(L)$ by applying $m-1$ band sum
  operations.  Then $K$ is slice.
\end{theoremalpha}

Our proof of Theorem~\ref{theorem:main-slice-knots} uses
Theorem~\ref{theorem:main-height-2} on Casson towers of height 2.
The details are discussed in Section~\ref{section:slice-knots}.

Theorem~\ref{theorem:main-slice-knots} specialises to several
interesting cases.  First, taking $L$ to be a knot we see that
Theorem~\ref{theorem:main-slice-knots} has, as a special case, the
result of Freedman that the Whitehead double of any knot is slice.
When $L$ is a knot there are no band sums.  More generally, when the
bands miss standard genus one Seifert surfaces for the components of
$\Wh(L)$, we obtain a knot of Alexander polynomial one, to which
Freedman's slicing result applies.  The novel aspect of
Theorem~\ref{theorem:main-slice-knots} is that arbitrary bands are
allowed.  By applying Theorem~\ref{theorem:main-slice-knots} for $L$ a
link and suitably complicated bands, we obtain a large family of new
slice knots.  For instance, the following corollary gives a way to
construct intriguing examples.

\begin{corollaryalpha}
  \label{corollary:main-whitehead-ribbon-slice}
  Suppose $L$ is an $m$-component homotopically trivial link, and $R$
  is a ribbon knot.  Consider a split union $\Wh(L)\sqcup R$ in $S^3$,
  and choose $m$ disjoint bands which join each component of $\Wh(L)$
  to~$R$, such that in addition the bands are disjoint from an
  immersed ribbon disc for $R$ in $S^3$ and are disjoint from Seifert
  surfaces for $\Wh(L)$.  Then the knot $K$ obtained from
  $\Wh(L)\sqcup R$ by these band sum operations along the arcs is
  slice.
\end{corollaryalpha}

The additional assumptions on the bands is not necessary to conclude
that $K$ is slice, but we include it so that we can discuss the ribbon
knot $R$ meaningfully.  For example, a slice knot $K$ from
Corollary~\ref{corollary:main-whitehead-ribbon-slice} has the same
Alexander polynomial as $R$; see
Proposition~\ref{proposition:properties-of-slice-knots}.  An explicit
example is given in Figure~\ref{figure:whitehead-ribbon-construction}.
To construct this knot apply band sum operations to $\Wh(L)\sqcup R$,
where $L$ is the 3-component link obtained from the Whitehead link by
adjoining an untwisted parallel of one of the components, and $R$ is
the ribbon knot~$8_{8}$.  By
Corollary~\ref{corollary:main-whitehead-ribbon-slice}, the knot $K$ in
Figure~\ref{figure:whitehead-ribbon-construction} is slice.  This $K$
is a hyperbolic knot (verified by SnapPea), and consequently is prime
and non-satellite.

\begin{figure}[t]
\includegraphics[scale=.8]{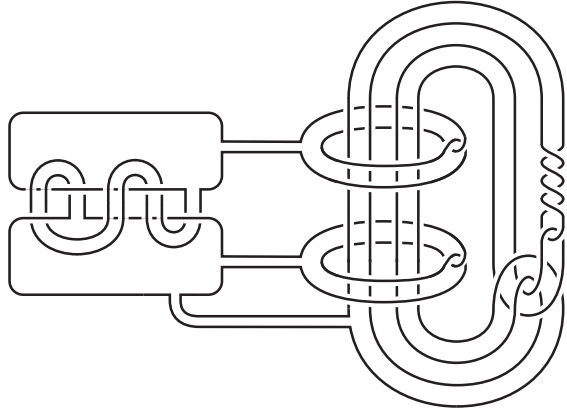}
\caption{An example of a new slice knot}
\label{figure:whitehead-ribbon-construction}
\end{figure}

To the knowledge of the authors, previously known methods and results
are not able to show that all of our knots are slice, except for in
some special cases.
Section~\ref{subsection:properties-of-slice-knots} contains more
details on the failure of the topological surgery method to slice
these knots.  Another possible approach to slice a knot $K$ given by
Corollary~\ref{corollary:main-whitehead-ribbon-slice} would be to show
that the \emph{link} $\Wh(L)\sqcup R$ is slice.  This is an important
conjecture in the theory of topological 4-manifolds.

\begin{conjecture}[Topological Whitehead Double Conjecture]
  \label{conjecture:topological-whitehead-double}
  The Whitehead double $\Wh(L)$ of a link $L$ is freely slice if and
  only if $L$ is homotopically trivial.
\end{conjecture}

Here a link is called \emph{freely slice} if there are slice discs
whose complement has free fundamental group.
Conjecture~\ref{conjecture:topological-whitehead-double} was stated
explicitly in~\cite[Conjecture~1.1]{Cochran-Friedl-Teichner:2006-1},
but was implicit in several earlier works such
as~\cite{Freedman:1988-2, Freedman-Lin:1989-1,
  Freedman-Teichner:1995-2}.

The only if direction of the conjecture implies that topological
surgery does not work for the rank two free
group~\cite{Casson-Freedman:1984-1, Freedman-Lin:1989-1}.  Freedman
confirmed the conjecture for knots and 2-component
links~\cite{Freedman:1988-2}.  See \cite{Krushkal:1999-1,
  Krushkal:2008-1, Krushkal:2014-1} for recent progress towards the
only if direction in the 3-component case.

The best known result toward the if direction of
Conjecture~\ref{conjecture:topological-whitehead-double} is a theorem
of Freedman and Teichner that if a link $L$ is homotopically
trivial$^+$, then $\Wh(L)$ is (freely)
slice~\cite{Freedman-Teichner:1995-2}, where $L$ is said to be
\emph{homotopically trivial\hspace{1pt}$^+$} if any link obtained from
$L$ by adjoining a zero-linking parallel copy of one of the components
is homotopically trivial.

Since Conjecture~\ref{conjecture:topological-whitehead-double} remains
open, this approach is not sufficient to slice the knots of
Corollary~\ref{corollary:main-whitehead-ribbon-slice} when one uses a
link $L$ which is homotopically trivial but not homotopically
trivial$^+$.  For instance, this is the case for the knot of
Figure~\ref{figure:whitehead-ribbon-construction}.

\subsubsection*{Homotopy ribbon-slice conjecture}

Recall that the ribbon slice conjecture claims that every slice knot
is a ribbon knot.  More precisely, the statement depends on the
category: in the smooth case, one asks whether a knot bounds a smooth
slicing disc if and only if it is a ribbon knot.  In the topological
case, first note that not all slice knots are ribbon, since there are
slice knots which are not smoothly slice.
Following~\cite{Casson-Gordon:1983-1}, we say that a knot $K$ in $S^3$
is \emph{homotopy ribbon} if there is a slicing disc $\Delta$ in $D^4$
for which the inclusion induces an epimorphism $\pi_1(S^3\sm K)
\twoheadrightarrow \pi_1(D^4\sm \Delta)$.  As stated in, for
instance,~\cite[Section~3.5]{Levine-Orr:1997-1} (also cf.\
\cite[Question~6.2]{Casson-Gordon:1983-1},
\cite[Problem~4.22]{Kirby:problem-list-1995-edition}), one asks the
following in the topological category.

\begin{conjecture}[Homotopy Ribbon Slice Conjecture]
  \label{conjecture:homotopy-ribbon-slice}
  A knot is slice if and only if it is homotopy ribbon.
\end{conjecture}

We remark that the homotopy ribbon property is essential in the study
of ribbon obstructions in the literature.  See for
example~\cite{Casson-Gordon:1983-1, Bonahon:1983-1,
  Casson-Gordon:1986-1, Miyazaki:1994-1, Friedl:2003-4,
  Davis-Naik:2002-1}.

Our geometric method, which applies the disc embedding theorem
directly to construct a slicing disc, gives potential counterexamples
to the homotopy ribbon slice conjecture.  In particular, we ask an
explicit question:

\begin{question}
  \label{question:homotopy-ribbon}
  Is the slice knot in
  Figure~\ref{figure:whitehead-ribbon-construction} homotopy ribbon?
\end{question}

We remark that no potential counterexample to the homotopy ribbon
slice conjecture was known---more precisely, every previously known
slice knot is either smoothly slice or homotopy ribbon, or both.  For
instance, many slice knots in the literature which are not smoothly
slice are knots with Alexander polynomial one. These are homotopy
ribbon by Freedman's result.  Slice knots obtained by the results in
\cite{Friedl-Teichner:2005-1} are homotopy ribbon, too.  Several
papers in the literature (e.g.\
\cite{Hedden-Livingston-Ruberman:2010-01, Cochran-Harvey-Horn:2012-1,
  Cochran-Horn:2012-1}) also consider slice knots produced by
satellite constructions using companion knots of Alexander polynomial
one.  They are all homotopy ribbon.  The essential reason for this is
that all use (building blocks obtained by) a topological surgery
construction of a slice disc exterior, with the fundamental group of a
ribbon disc exterior, as in \cite{Freedman-Quinn:1990-1}
and~\cite{Friedl-Teichner:2005-1}.  The smoothly slice knots presented
in \cite{Gompf-Scharlemann-Thompson:2010-1,Abe-Tange:2013-1}, which
are not known to be ribbon, can be seen to be homotopy ribbon by
inspecting their constructions.

We remark that the if direction of
Conjecture~\ref{conjecture:topological-whitehead-double} would imply
that the slice knots given by
Corollary~\ref{corollary:main-whitehead-ribbon-slice} are homotopy
ribbon, and in particular that the answer to
Question~\ref{question:homotopy-ribbon} is yes.  It is also
interesting to note that according
to~\cite[Proposition~2]{Casson-Freedman:1984-1}, in order to solve
4-dimensional surgery problems, one needs the Whitehead doubled links
in Conjecture~\ref{conjecture:topological-whitehead-double} to be
homotopy ribbon.  We discuss more related questions at the end of
Section~\ref{section:slice-knots}.

\subsubsection*{Smooth status}
The smooth status of our knots is also interesting.  We think our
knots are unlikely to be smoothly slice (particularly when all the
clasps have the same sign); compare~\cite{LevineA:2009-01}.  For some
special cases of
Corollary~\ref{corollary:main-whitehead-ribbon-slice}, we computed the
Rasmussen $s$-invariant to be nonzero, aided by a computer.  Thus at
least some of our examples are not smoothly slice.  Recall that the
Alexander polynomial satisfies $\Delta_K(t) = \Delta_R(t)$, in the
notation of Corollary~\ref{corollary:main-whitehead-ribbon-slice}\@.
The following natural question arises.

\begin{question}
  \label{question:smooth-conc-alex-poly-one}
  Is the slice knot in
  Figure~\ref{figure:whitehead-ribbon-construction} smoothly
  concordant to an Alexander polynomial one knot?
\end{question}

In particular for the knot $K$ of
Figure~\ref{figure:whitehead-ribbon-construction} we have
\[
\Delta_K(t) = \Delta_{8_8}(t) = 2t^{-2}-6t^{-1}+ 9-6t+ 2t^2 = (2t^2 -
2t+1)(2t^{-2} -2t^{-1} +1).
\]
We think the answer to
Question~\ref{question:smooth-conc-alex-poly-one} is likely to be no,
but we do not know at present how to perform the computation of
$d$-invariants which we think will be necessary to prove this.  The
existence of topologically but not smoothly slice knots with this
property was shown in~\cite{Hedden-Livingston-Ruberman:2010-01}.  We
remark that their examples were constructed using satellite operations
which tied in Alexander polynomial one knots.  We conjecture that
there are slice knots produced by
Corollary~\ref{corollary:main-whitehead-ribbon-slice}, which are
linearly independent from the examples
in~\cite{Hedden-Livingston-Ruberman:2010-01} in the smooth knot
concordance group modulo Alexander polynomial one knots.

\subsubsection*{New slice links}

Using (distorted) Casson towers of height 4 and
Theorem~\ref{theorem:main}, we prove the following two results on
links.  To state the first, we consider the following operation, which
is called \emph{ramified Whitehead doubling}: for a given knot, take
some number of untwisted parallel copies, and then replace each
parallel copy by its untwisted Whitehead double.  Either sign may be
used for the clasp.  We may iterate, by applying this operation again
to each component produced in the previous step.  If we repeat this
$n$ times, where the number of parallel copies used in each iteration
may change, then we say that the result is obtained by an
\emph{$n$-fold ramified Whitehead doubling}.  Define a \emph{ramified
  $\Wh_n$ link} to be a link obtained from the Hopf link by replacing
one component with its $n$-fold ramified Whitehead double.

\begin{theoremalpha}
  \label{theorem:main-ramified-wh4}
  Any ramified $\Wh_n$ link is slice for $n\ge 4$.
\end{theoremalpha}

Freedman showed that the unramified $\Wh_3$, and consequently
unramified $\Wh_n$ for $n\ge 3$, are slice~\cite{Freedman:1988-2}.
Although Theorem~\ref{theorem:main-ramified-wh4} is for $n\ge 4$ only,
the ramified $n=4$ case gives new slice links.  The case $n \ge 5$ was
shown by Gompf-Singh~\cite{gompf-singh-1984}.

As a second application of Theorem~\ref{theorem:main}, we prove the
following:

\begin{theoremalpha}
  \label{theorem:main-distorted-casson-tower-link-example}
  The link in Figure~\ref{figure:distorted-tower-slice-link} is slice.
\end{theoremalpha}

In fact, more generally, we specify a class of slice links, related to
distorted Casson towers of height 4, which contains the link in
Figure~\ref{figure:distorted-tower-slice-link}.  See
Section~\ref{subsection:distorted-iterated-ramified-whitehead-double},
and particularly Theorem~\ref{theorem:distorted-casson-tower-slicing}.

\begin{figure}[t]
  \labellist
  \footnotesize
  \pinlabel{\rotatebox{20}{$-2$}} at 61 185
  \pinlabel{\rotatebox{65}{$-2$}} at 11 150
  \endlabellist
  \includegraphics[scale=.82]{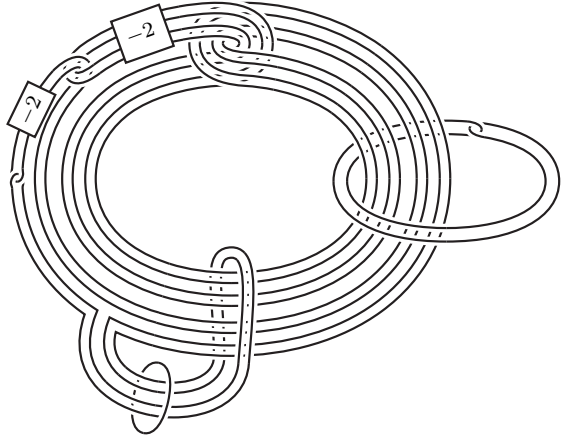}
  \caption{A slice link obtained from a distorted Casson tower of
    height~4.  Each box labeled $-2$ designates two left-handed full
    twists.}
  \label{figure:distorted-tower-slice-link}
\end{figure}


We finish the introduction with a couple of additional remarks.  Our
proofs of slicing results are more geometric and direct than many
previous applications of the disc embedding theorem to slicing knots
and links.  Our method is similar in character to Kervaire and
Levine's programme~\cite{Kervaire:1965-1, Levine:1969-1} for high
dimensional knots, in that we construct the \emph{slice discs}
in~$D^4$, while most previous slicing results (e.g.\
\cite{Freedman:1988-2, Freedman-Teichner:1995-2,
  Friedl-Teichner:2005-1, Cochran-Friedl-Teichner:2006-1,
  Davis:2006-1}) proceed by first constructing a slice disc
\emph{exterior} using 4-dimensional topological surgery, following the homology surgery slicing strategy of
Cappell-Shaneson~\cite{Cappell-Shaneson:1974-1}.  The only exception
known to the authors is an alternative proof of Garoufalidis-Teichner
that Alexander polynomial one knots are
slice~\cite{Garoufalidis-Teichner}.  Of course there is a reason for
this: the surgery method is often remarkably effective.

Recall that while one can visualise a ribbon disc, (for example by
drawing ``movie pictures'' of cross sections), it is nigh on
impossible to visualise a slice disc for a knot or link which is
topologically but not smoothly slice.  Using the slicing theorems of
this paper, one can at least understand a neighbourhood, in the
4-ball, used in the construction of the disc.  The reader may also
perhaps see it as virtuous, when constructing slice discs, to minimise
the number of times Freedman's disc embedding theorem is used; the
slicing constructions of this paper only use it once per slice disc.
This can be contrasted with the number of Freedman discs required to
employ topological surgery and $h$-cobordism when using the homology
surgery method, or even in
Garoufalidis-Teichner~\cite{Garoufalidis-Teichner}.

\subsubsection*{Organisation of the paper}

In Section~\ref{section:definitions} we give preliminary definitions
of Casson towers, gropes, and related objects.  In
Section~\ref{section:from-gropes-to-disc}, we prove that the existence
of a height 1.5 grope with a certain fundamental group condition gives
rise to a flat embedded disc with the desired boundary.  In
Section~\ref{section:casson-towers} we prove our main disc embedding
results for Casson towers: Theorems~\ref{theorem:main},
\ref{theorem:main-height-3}, and~\ref{theorem:main-height-2}\@.
Sections~\ref{section:slice-knots} and~\ref{section:slice-links}
present new slice knots and links respectively.
Section~\ref{section:casson-tower-3-infinite-grope} discusses the
grope filtration of knots and Casson towers.

\subsubsection*{Acknowledgements}

The authors would like to thank Kent Orr and Peter Teichner for some
very helpful conversations and suggestions.  Wojciech Politarczyk and
the second author worked together on the combinatorics chapter for the
Freedman lecture notes, from which the proof of
Lemma~\ref{lemma:grope-height-raising} is derived, and the second
author gained a great deal of understanding from this collaboration.
We also thank the referees for their very useful comments.  The first
author was partially supported by NRF grants 2013067043 and
2013053914.

\section{Preliminary definitions}\label{section:definitions}

All 4-manifolds in this paper are compact and oriented.  First we give
the definitions of a Casson tower and a grope, which are the two main
objects we will be working with.

\begin{definition}[\cite{Casson-1986-towers}]
  \label{definition:casson-tower}
  A \emph{Casson tower} is a 4-manifold $T$ with a framed embedded
  circle $C=C(T)$ in its boundary, defined as below.
   We write $\partial T$ as a union of
  two codimension zero submanifolds $\partial_+$ and $\partial_-$,
  where $\partial_-$ is a (closed) tubular neighbourhood of $C$ in
  $\partial T$ and $\partial_+ = \overline{\partial T\sm \partial_-}$.
  We call $C$ and $\partial_-$ the \emph{attaching circle} and the
  \emph{attaching region} respectively.

  A Casson tower has a height $n \in \mathbb{N}$.  A \emph{Casson
    tower $T_1$ of height 1}, which we will also call a \emph{plumbed
    handle}, is a thickened disc $D^2 \times D^2$, with
  $C:=\partial D^2\times 0$, $\partial_- := \partial D^2 \times D^2$,
  and with some number of self-plumbings performed in the interior of
  $D^2\times 0$.  A self-plumbing is performed by taking two discs
  $D_1, D_2 \subset \Int D^2\times 0$ and then
  identifying $D_1\times D^2$ and $D_2\times D^2$, viewing each 2-disc
  as the unit disc in~$\mathbb{C}$, via $(z,w)\sim (w,z)$ to produce a
  positive self-intersection of $D^2\times 0$, or
  $(z,w)\sim (\bar w, \bar z)$ to produce a negative
  self-intersection.  The \emph{core} disc is defined to be the image
  of $D^2\times 0$ in $T_1$, which is now an immersed disc.  The
  attaching circle $C$ is framed by the restriction of the unique
  framing of $D^2\times 0$ \emph{before} plumbing.  Equivalently, if
  we perform $k_+$ positive and $k_-$ negative plumbings, then the
  framing on $C$ is obtained by twisting the restriction to $C$ of the
  unique framing of the core disc in $T_1$ by $k_--k_+$.  To each
  double point of the core disc, there is an associated \emph{double
    point loop} on the core disc, which departs the double point into
  a sheet and comes back through the other sheet, avoiding all other
  double points.  Isotope this loop to obtain an embedded circle in
  $\partial_+$, which we call a \emph{double point loop
    in~$\partial_+$}.  We assume that double point loops are disjoint.
  Double point loops are framed in such a way
  there is a diffeomorphism of the plumbed handle with 2-handles
  attached along this framing to the 4-ball which takes the attaching
  circle to the zero framed unknot in the 3-sphere.  There is a unique
  such framing.

  The framings can be explicitly described using a standard Kirby
  diagram of a plumbed handle, shown in
  Figure~\ref{fig:plumbed-handle-diagram}: the circles $a_i$ are
  double point loops on $\partial_+$, and the framings on $C$ and
  $a_i$ described above are the zero-framing in
  Figure~\ref{fig:plumbed-handle-diagram} (see \cite[Lemma~2 of
  Lecture~I]{Casson-1986-towers}).

  \begin{figure}[H]
    \labellist
    \small
    \pinlabel {$C$} at 14 7
    \pinlabel {$a_1$} at 15 87
    \pinlabel {$a_2$} at 78 87
    \pinlabel {$a_k$} at 158 87
    \endlabellist
    \includegraphics[scale=1]{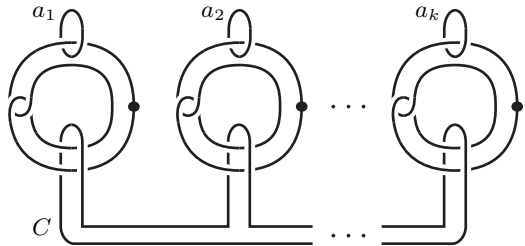}
    \caption{A standard Kirby diagram of a plumbed handle, together
      with the attaching circle $C$ and double point loops $a_i$ on
      the boundary of the plumbed handle.  Each plumbing corresponds
      to a dotted circle, where the sign of the plumbing determines
      the sign of the clasp.}
    \label{fig:plumbed-handle-diagram}
  \end{figure}

  For $n \geq 2$, a Casson tower of height $n$, denoted $T_n$, is
  constructed inductively by taking a height one Casson tower and, for
  each double point loop, identifying a neighbourhood of the double
  point loop on $\partial_+$ with the attaching region of some Casson
  tower $T_{n-1}$ of height $n-1$, along the preferred framings of the
  double point loop and $C(T_{n-1})$. A different height $n-1$ Casson
  tower may be used for each double point loop.  The attaching circle
  $C(T_n)$ and the attaching region $\partial_-$ of the new Casson
  tower $T_n$ are just those of the original height one tower.
\end{definition}

\begin{definition}\label{definition:stages-notation}
  The \emph{$k$th stage} of a Casson tower $T_n$ is the material that
  was introduced by the $(n-k+1)$th inductive step in the
  construction, where taking a height one tower counts as the first
  step.  Following~\cite{Freedman:1982-1}, we denote the union of the
  stages of a Casson tower $T$ from $p$ through $q$ inclusive
  by~$T_{p\ed q}$.
\end{definition}

\begin{definition}[\cite{Freedman-Quinn:1990-1,
    Freedman-Teichner:1995-1}]
  \label{definition:grope}
  A \emph{grope} of height $n$ ($n\in \mathbb{N}$) is a pair
  (2-complex, base circles) of a certain type described inductively
  below.  A \emph{grope of height $1$} is a disjoint union of oriented
  connected surfaces each of which has connected nonempty boundary.
  The boundary circles are the base circles.  Take a grope $G_1$ of
  height 1, and let $\{\alpha_1,\ldots,\alpha_{2g}\}$ be a standard
  symplectic basis of circles for the first homology of~$G_1$. Then a
  \emph{grope of height $n+1$} is formed by, for each $i$, attaching
  some grope of height $n$ with a single base circle to $G_1$,
  identifying the base circle of this height $n$ grope
  with~$\alpha_i$.  A different height $n$ grope may be used for each
  $\alpha_i$.  The base circles of the height $n+1$ grope are defined
  to be the base circles of $G_1$.  We often call the base circles the
  \emph{boundary} of the grope.  The \emph{$k$th stage} of the grope
  is the union of the surfaces that were introduced by the $(n-k+1)$th
  inductive step in the construction, where taking a height one grope
  counts as the first step.  We denote a grope of height~$n$ by~$G_n$,
  and denote the stages from~$p$ through~$q$ inclusive by~$G_{p\ed
    q}$.

  A \emph{capped grope} is constructed by attaching a disc to each of
  a symplectic basis of curves for the top stage surfaces.  These
  discs are referred to as the \emph{caps}.  The surface stages are
  called the \emph{body} of the grope.  A capped grope of height $n$,
  sometimes also known as a \emph{capped grope with $n$~surface
    stages}, will be denoted $G_n^c$.
\end{definition}

We note that, in this paper, a capped grope has a multi-component body
in general.  It is called a \emph{union-of-discs-like} capped grope
in~\cite{Freedman-Quinn:1990-1}.  As a special case, if the body is
connected, it is called a \emph{disc-like} capped grope.  A 2-complex
obtained by attaching 2-discs to capped gropes along each boundary
component is called a \emph{union-of-spheres-like} capped grope.
These only differ from union-of-discs-like capped gropes in that the
bottom stage surfaces are closed.  In this paper a capped grope refers
by default to a union-of-discs-like capped grope.  On the other hand,
in the sequel, we will construct ``transverse'' capped gropes by
glueing together union-of-discs-like capped gropes.  Transverse capped
gropes are always union-of-spheres-like.

As shown in Figure~\ref{figure:grope-example}, a capped grope has a
standard model embedded in~$\R^3$.

\begin{figure}[H]
  \includegraphics[scale=.95]{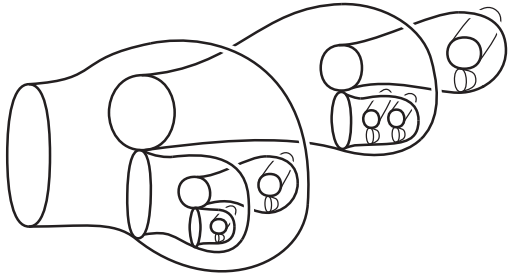}
  \caption{A standard model of a capped grope of height 3 in~$\R^3$.}
  \label{figure:grope-example}
\end{figure}

\begin{definition}
  \label{definition:framing-of-gropes}
  Start with a model (capped) grope embedded in $\R^3$, and embed this
  in $\R^4$ via $\R^3 \hookrightarrow \R^3 \times \R \cong \R^4$.
  Take a thickening of the model in $\R^4$.  We refer to this
  thickening as a \emph{framed (capped) grope}, which is a compact
  4-manifold with boundary.  The \emph{attaching region} $\partial_-$
  of a framed (capped) grope is the compact 3-manifold with boundary
  given by thickening the base circles of the bottom surface stage.
  The attaching region is a submanifold of the boundary of the framed
  (capped) grope.
\end{definition}

Note that each surface and disc component of a model (capped) grope in
$\R^4$ has a canonical framing of its normal bundle given by taking a
1-dimensional framing in $\R^3$, and extending via the trivial line
bundle when we take the product with~$\R$.

We always regard the base circles of a framed (capped) grope as framed
circles endowed with the induced framing.  Similarly, the symplectic
basis curves of the top stage surfaces of a framed grope have an
induced framing.  In the case of a framed capped grope, the framing
for the symplectic basis curves is equal to the restriction of the
unique framing of the caps.

The next definition is of proper immersions.  Briefly, a properly
immersed grope has embedded body, and immersed caps which are disjoint
from the body.

\begin{definition}[\cite{Freedman-Quinn:1990-1}]
  \label{definition:proper-immerssed-capped-grope}
  Take a model framed capped grope as in
  Definition~\ref{definition:framing-of-gropes}, and introduce
  plumbings into the model, by plumbing together caps and introducing
  self-plumbings in caps.  A \emph{proper immersion} of a capped grope
  into a 4-manifold $M$ with boundary is an embedding of this plumbed
  model into $M$ such that the attaching region $\partial_-$ maps
  to~$\partial M$.
\end{definition}

We remark that for a proper immersion it is allowed to plumb two caps
attached to any, possibly different, body components.

We will also denote a framed grope of height~$n$ by~$G_n$, and a
framed capped grope of height $n$ by~$G_n^c$.  From now on when we
refer to a capped grope, we will mean a framed capped grope, often a
properly immersed framed capped grope.  However we will also refer to
geometric operations, such as tubing and taking parallel copies, on
surfaces which are part of the underlying 2-complex, hopefully without
causing confusion.  When there is concern about ambiguity, we will
denote a (further) thickening of a (framed) capped grope by~$\nu
G^c_n$.

We will be interested in improving a capped grope to a \emph{one storey
  capped tower}.  Briefly, a proper immersion of a one storey capped
tower is a capped grope with caps for the caps, that is, discs bounded
by the double point loops of the caps.  The first layer of caps should
have self-intersections only.  The second layer of caps, called the
\emph{tower caps}, must be disjoint from the body \emph{and} the caps
of the capped grope.  The official definition is next.

\begin{definition}[\cite{Freedman-Quinn:1990-1}]
  \label{definition:one-story-capped-tower}
  A \emph{one storey capped tower} is obtained from a model framed
  capped grope by performing finger moves that introduce
  self-plumbings into the caps, and then by adjoining disjoint Whitney
  discs and accessory discs.  We call the Whitney discs and accessory
  discs the \emph{tower caps}.  A \emph{proper immersion} of a one
  storey capped tower into a 4-manifold~$M$ is obtained by introducing
  plumbings (not necessarily just self-plumbings) into the tower caps
  and then embedding the plumbed model into~$M$.  Note that tower caps
  still miss the entire capped grope.
\end{definition}

A reference for finger moves, Whitney discs, and accessory discs for
readers not familiar with them is \cite[Sections~1.5
and~3.1]{Freedman-Quinn:1990-1}.  In this paper, we do not need their
definitions nor do we need any properties of tower caps, since we are
not going to work with capped towers. Rather, once we have one we will
observe (Lemma~\ref{lemma:capped-grope-to-one-story-tower} and
Theorem~\ref{theorem:one-story-tower-four-stages-implies-disc}) that
this is sufficient to activate the Freedman-Quinn machine and produce
an embedded disc.

\section{Obtaining a disc from a framed capped grope}
\label{section:from-gropes-to-disc}

This section proves Theorem~\ref{theorem:grope-to-disc-main} below,
which is our main technical result.  It sharpens the minimal grope
data required to produce a disc, and will be used to deduce all of our
various statements on Casson towers.  In order to state the theorem we
will introduce some terminology and background.



\begin{definition}[\cite{Freedman-Teichner:1995-1}]
  \label{definition:good-group}
  We say that a group $\pi$ is \emph{good} if for any properly
  immersed disc-like capped grope $G^c$ of height 2 (obtained from a
  model framed capped grope by plumbing) and for any homomorphism
  $\pi_1(G^c) \to \pi$, there is an immersed disc in $G^c$ whose
  framed boundary is equal to the base circle of~$G^c$ and whose
  double point loops have trivial image in~$\pi$.


\end{definition}


In the book of Freedman and Quinn~\cite{Freedman-Quinn:1990-1}, a
different definition is used. They call a group good if the disc
embedding theorem \cite[p.~5]{Freedman-Quinn:1990-1} holds.
Definition~\ref{definition:good-group}, which is from
\cite{Freedman-Teichner:1995-1}, describes a (potentially) smaller
class of groups than the definition of~\cite{Freedman-Quinn:1990-1}.
A good group in the sense of Definition~\ref{definition:good-group} is
sometimes called a \emph{Null Disc Lemma (NDL) group}; see for
example~\cite{Kirby-Taylor:2001-1}.


\begin{definition}
Following \cite[Section~3]{Krushkal-Quinn:2000-1}, a group $\pi$ has
\emph{subexponential growth} if given any finite subset
$S \subseteq \pi$, there is an integer~$n$ such that the set of all
products of elements of~$S$ with length~$n$ determine fewer than $2^n$
elements of the group~$\pi$.
\end{definition}

\begin{theorem}[\cite{Freedman-Teichner:1995-1,Krushkal-Quinn:2000-1}]
  \label{theorem:krushkal-quinn}
  Any group of subexponential growth is a good group.
\end{theorem}

In addition, the class of good groups is closed under extensions and
direct limits~\cite[Lemma~1.2]{Freedman-Teichner:1995-1}.

For $n \in \mathbb{N}$, a \emph{capped grope of height $n.5$} is
constructed by attaching a height $n$ capped grope to one curve from
each dual pair of curves in a symplectic basis for the first homology
of a height one grope, and then attaching height $n-1$ capped gropes
to the remaining curves.  (By convention, a height $0$ capped grope is
a union of discs.)
A \emph{proper immersion} is defined by allowing plumbings of the
caps, similarly to the height~$n$ case.

Now we state the main theorem of this section.  The theorem is
probably known to the experts, but a detailed proof has not appeared.

\begin{theorem}[Disc Embedding for Capped Gropes of Height $\ge 1.5$]
  \label{theorem:grope-to-disc-main}
  For some $n \in \frac{1}{2}\mathbb{N}$ which is at least $1.5$, let
  $(G^c_{n}, \partial_-) \rightarrow (M,\partial M)$ be a properly
  immersed capped grope of height $n$ in a 4-manifold $M$, and let $\nu
  G^c_n$ be a further thickening of $G^c_n$. Suppose that the image of
  the inclusion induced map
  \[
  \pi_1(\nu G^c_n \sm G_{1\ed 1}, *) \to \pi_1(M\sm G_{1\ed 1}, *)
  \]
  is a good group, for all choices of basepoint $*$ in
  $\nu G^c_n \sm G_{1\ed 1}$.  Then there are disjoint flat embedded
  discs in $M$ with the same framed boundary as~$G^c_{n}$.
\end{theorem}

Recall that the body of a capped grope need not be connected.  We
also remark that in order to check that the hypothesis is true for all
basepoints, it is enough to check it for some choice of basepoint in
each connected component of $\nu G^c_n \sm G_{1\ed 1}$.

Next we state some results from the literature which will be used
during the proof of Theorem~\ref{theorem:grope-to-disc-main}.  In the
following lemma, the phrase ``$G_n^c$ has transverse spheres'' means
that each connected component of the bottom stage of $G_n^c$ has a
transverse sphere intersecting $G_n^c$ in precisely one point, which
lies in that connected component of the bottom stage.  Also, we say
that a properly immersed capped grope $G_n^c$ in $M$ is
\emph{$\pi_1$-null} if any loop in the image of $G_n^c$ is
null-homotopic in~$M$.

\begin{lemma}[{\cite[Section~3.3]{Freedman-Quinn:1990-1}}]
  \label{lemma:capped-grope-to-one-story-tower}
  Suppose $G_n^c \to M$ is a proper immersion of a capped grope of
  height $n$ at least 2, into a 4-manifold $M$, which is $\pi_1$-null
  and has transverse spheres.  Then the embedding of the body of
  $G_n^c$ extends to a proper immersion of a union-of-discs-like one
  storey capped tower with arbitrarily many surface stages.
\end{lemma}

The first step in the proof of
Lemma~\ref{lemma:capped-grope-to-one-story-tower} given in
\cite[Section~3.3]{Freedman-Quinn:1990-1} uses grope height raising
(see Lemma~\ref{lemma:grope-height-raising} below) to find a capped
grope with arbitrary height in a neighbourhood of the given capped
grope.  This capped grope is then improved to a one storey capped
tower.  Freedman and Quinn's statement begins with a proper immersion
of a capped grope of height at least~3.  In our statement we have
replaced height~3 with height~2, in light of
Lemma~\ref{lemma:grope-height-raising} below.

The strategy of the proof of Theorem~\ref{theorem:grope-to-disc-main}
will be to arrange a situation where
Lemma~\ref{lemma:capped-grope-to-one-story-tower} can be applied.  We
will then be able to apply the following theorem
of~\cite{Freedman-Quinn:1990-1}.

\begin{theorem}[Freedman-Quinn]
  \label{theorem:one-story-tower-four-stages-implies-disc}
  A neighbourhood of a properly immersed one storey capped tower with
  at least four surface stages contains an embedded flat
  topological disc with the same framed boundary.
\end{theorem}

\begin{proof}
  Follow the arguments of the second and third sentences of
  \cite[Proof of Theorem~5.1A]{Freedman-Quinn:1990-1}.  The given data
  is sufficient to perform \emph{tower height raising with control}.
  Begin with the tower height raising proposition in
  \cite[Section~3.5]{Freedman-Quinn:1990-1}, and introduce control, to
  produce an infinite convergent tower as
  in~\cite[Sections~3.6--8]{Freedman-Quinn:1990-1}.  See in particular
  \cite[Proposition~3.8]{Freedman-Quinn:1990-1}.  A convergent
  infinite tower is then shown to be homeomorphic, relative to its
  attaching region $\partial_-$, to an open 2-handle via ``the
  design'' and Bing shrinking~\cite[Chapter~4]{Freedman-Quinn:1990-1}.
  See \cite[Theorem~4.1]{Freedman-Quinn:1990-1}.
\end{proof}

A key step in the proof of Theorem~\ref{theorem:grope-to-disc-main} is
to perform \emph{grope height raising}.

\begin{lemma}[Grope Height Raising Lemma]
  \label{lemma:grope-height-raising}
  Let~$G_{1.5}^c$ be a height $1.5$ capped grope which is properly
  immersed in a 4-manifold. For any $n \in \mathbb{N}$ there exists,
  inside a neighbourhood of~$G_{1.5}^c$, a properly immersed
  height~$n$ capped grope~$G_n^c$ with the same framed boundary.
\end{lemma}

The statement that a height~$1.5$ capped grope can have its height
raised arbitrarily is not contained in \cite{Freedman-Quinn:1990-1}
(the best statement given is in an exercise in their Section~2.7, and
involves height $2.5$).  The outline of the proof of the Grope Height
Raising Lemma was explained to the second author by F.~Quinn and
P.~Teichner, in the discussion sessions associated to the MPIM lecture
series of M.~Freedman~\cite{teamfreedman}.  For the convenience of the
reader we include the details below, after stating and proving
Lemma~\ref{lemma:cap-separation}.

Consider a framed capped grope $G_{n}^c$. Divide the surfaces
(including the caps) above the first stage into two sides, labelled as
the $+$ and $-$ sides, as follows. For a dual pair of curves in a
symplectic basis for the first stage surface, the surface attached to
one curve is labelled $+$ and the surface attached to the dual curve
is labelled~$-$.  A surface of stage 3 or higher has the same label as
that of the surface to which it is attached.  We therefore have $+$
and $-$ side height $n-1$ capped gropes.  When beginning with a height
$n.5$ capped grope, choose labels so that we have $+$ side height $n$
capped gropes and $-$ side height $n-1$ capped gropes.  In particular,
for height $1.5$, we just have caps on the $-$ side.

The following lemma is an important preliminary construction in the
height raising process.  When starting with a grope of height at
least~3, this step can be avoided by using the argument
of~\cite{Freedman-Quinn:1990-1}.  When starting with a grope of height
$1.5$ or $2$ however, the Cap Separation
Lemma~\ref{lemma:cap-separation} below seems to be necessary.

\begin{lemma}[Cap Separation Lemma]\label{lemma:cap-separation}
  For any $n \in \frac{1}{2}\mathbb{N}$ with $n \geq 1.5$, within a
  neighbourhood of a height $n$ capped grope, there is a height $n$
  capped grope with the same framed boundary and with the $+$ side
  caps disjoint from the $-$ side caps.
\end{lemma}

We believe some of the details of the proof we give to be new.

In the arguments below we will use the \emph{symmetric contraction}
described in~\cite{Freedman-Quinn:1990-1}.  For readers not familiar
with this notion, we state the key properties. Given a properly
immersed capped surface $\Sigma^c$ (i.e.\ a height~1 capped grope) in
a 4-manifold, \cite[Section~2.3]{Freedman-Quinn:1990-1} associates an
immersed genus zero surface $\Sigma'$ in a given neighbourhood of
$\Sigma^c$, with $\partial\Sigma' = \partial\Sigma^c$, such that any
immersed surface $S$ disjoint from the body of $\Sigma^c$ is regularly
homotopic to a surface $S'$ disjoint from $\Sigma'$, via a homotopy
supported in a given neighbourhood of the union of the caps
of~$\Sigma^c$.  We say that $\Sigma'$ is obtained from $\Sigma^c$ by
\emph{symmetric contraction}, and $S'$ is obtained by \emph{pushing
  $S$ off the contraction}.  We remark that $S'$ may have additional
self-intersections; more generally, for two such surfaces $S_1$ and
$S_2$, the pushed-off surfaces $S_1'$ and $S_2'$ may intersect in more
points than their antecedents $S_1$ and~$S_2$.

In the remainder of this paper, unless specified otherwise,
`contraction' and the verb `to contract' refer by default to the
symmetric contraction.

\begin{proof}[Proof of Cap Separation Lemma~\ref{lemma:cap-separation}]
  We need to remove intersections between the $+$ and $-$ side caps.
  Let $F_-$ denote the transverse capped grope for the $-$ side, which
  is constructed from two parallel copies of the $+$ side capped grope
  and an annulus which joins them together in a neighbourhood of the
  attaching circle for the $+$ side; see
  \cite[Section~2.6]{Freedman-Quinn:1990-1}.  Note that the number of
  components of $F_-$ is equal to the genus of the bottom stage
  surface.

  Contract the top stage of $F_-$ and push the $-$ side caps off the
  contraction.  Note that when we push off the contraction we may
  obtain new intersections of the $-$ side caps, but that is
  acceptable.  Contract further if necessary, to obtain a collection
  of transverse spheres $F'_-$, each of which is dual to a $-$ side
  stage 2 surface or cap.  We have that $F'_-$ is disjoint from the
  $-$ side apart from the transverse point, but $F'_-$ may intersect
  the $+$ caps, since we did not push those off the contraction.  For
  a given $-$ cap we want one of these transverse spheres in order to
  eliminate intersections with the $+$ caps; for each such
  intersection, push the intersection down to a $-$ side surface of
  stage 2 if necessary, and tube the $+$ side cap into a parallel copy
  of the $F'_-$ transverse sphere.  We may obtain new intersections
  between $+$ side caps, since $F'_-$ may intersect $+$ side caps, and
  from the fact that $F'_-$ will probably not be embedded.  But this
  is acceptable too.
\end{proof}

We are now ready to give the proof of
Lemma~\ref{lemma:grope-height-raising}.  The proof uses arguments
known to the experts; a variant appeared in
\cite[Section~2.7]{Freedman-Quinn:1990-1}.  Ours is based on that
typed up by the second author and W.~Politarczyk in
~\cite[Section~8.3]{teamfreedman}, which itself derived from the
Freedman lectures and discussion sessions.

\begin{proof}[Proof of Grope Height Raising
  Lemma~\ref{lemma:grope-height-raising}]

  By the Cap Separation Lemma~\ref{lemma:cap-separation}, we may
  suppose that we have a properly immersed grope of height $1.5$ where
  the $+$ caps are disjoint from the $-$ caps.

  Let $F_-$ denote the transverse capped surface for the $-$ side,
  which is constructed from two parallel copies of the $+$ side capped
  surface.  Tube each intersection of the $-$ side caps into a
  parallel copy of $F_-$.  This turns the $-$ side caps into capped
  surfaces.  We have now raised the height by one on the $-$ side.  We
  started with $(+,-)$ heights being $(1,0)$ and now we have $(1,1)$.

  Next we raise height on the $+$ side.  To achieve this we repeat the
  above process, with $+$ and $-$ reversed.  That is, first we apply
  Lemma~\ref{lemma:cap-separation} to once again separate the caps on
  the $+$ and $-$ sides.  Then we let $F_+$ denote the transverse
  capped surface for the new $+$ side.  Tube intersections of $+$ side
  caps into parallel copies of $F_+$.  This creates a grope with
  $(+,-)$ height $(2,1)$.

  Now repeat as many times as desired, alternating $+$ and $-$ sides,
  but with $F_{\pm}$ as transverse capped gropes instead of capped
  surfaces.  From $(2,1)$ we go to $(2,3)$, then $(5,3)$, then $(5,8)$
  and so on.  The $(+,-)$ heights grow according to the Fibonacci
  numbers.  We remark that we could instead apply the argument in
  \cite[Section~2.7]{Freedman-Quinn:1990-1} once we get both sides of
  height at least $2$; it raises height a little slower but creates
  fewer intersections along the way.

  To finish, once both sides have height at least $n$, contract until
  both sides have height exactly~$n$, so that we have a height~$n$
  symmetric grope.
\end{proof}

\begin{remark}\label{remark:grope-height-raising-from-n}
  The proof of Lemma~\ref{lemma:grope-height-raising} can be applied
  to raise the height of a grope of any height $n \ge 2$.  In this
  case one obtains a Fibonacci sequence beginning with $(n-1,n-1)$.
\end{remark}

With the preliminaries at last complete, there now follows the proof
of Theorem~\ref{theorem:grope-to-disc-main}.

\begin{proof}[Proof of Theorem~\ref{theorem:grope-to-disc-main}]
  Our goal is to extend our capped grope to a properly immersed one
  storey capped tower, so that we can apply
  Theorem~\ref{theorem:one-story-tower-four-stages-implies-disc}.

  We may assume that $n=1.5$, since the first step is to perform
  \emph{grope height raising}.  (One can either first contract the
  grope until it is of height~$1.5$, or use
  Remark~\ref{remark:grope-height-raising-from-n}.)
  In each step below, first we state what to do, and then discuss how.

  \begin{step-named}[Step 1]
    Extend the body of $G^c_{1.5}$ to a properly immersed framed
    capped grope $G^c_{6}$ of height $6$ in~$\nu G^c_{1.5}$.
  \end{step-named}

  This is done by applying Grope Height Raising
  Lemma~\ref{lemma:grope-height-raising}, to raise the height of
  $G^c_{1.5}$ by~$4.5$.

  \begin{step-named}[Step 2]
    Extend $G_{1\ed 5}\subset G^c_{6}$ to a height 5 properly immersed
    framed capped grope $G^c_{5}$ in a neighbourhood of $G^c_{6}$,
    such that all of the cap intersections of $G^c_{5}$ are
    self-intersections.
  \end{step-named}

  The argument for Step 2 is as follows.  Contract the top stage
  surfaces of $G_6^c$ one at a time, inductively, pushing caps of all
  the other remaining top stage surfaces off the contraction before
  contracting the next top stage surface.  The result is a height 5
  capped grope $G^c_5$ with the desired property.

  \begin{step-named}[Step 3]
    Modify $G^c_{5}$ in such a way that each stage 2 surface $\Sigma$
    of $G^c_{5}$ has a transverse sphere which meets $G^c_{5}$ at
    exactly one point, and that point belongs to~$\Sigma$; also, caps
    attached to distinct top level surfaces are disjoint.
  \end{step-named}

  This is achieved using what is by now a standard argument
  (\cite[Section~2.6]{Freedman-Quinn:1990-1}), as follows.  We use the
  following notation: for a surface $\Sigma$ in the body of a capped
  grope $G^c$, let $G^c_{\Sigma}$ be the capped grope consisting of
  the surfaces and caps on the top of $\Sigma$, including $\Sigma$
  itself.  Let $\Sigma$ be a stage~2 surface of $G^c_{5}$, and let
  $\Sigma'$ be the stage~2 surface dual to~$\Sigma$ i.e.\ the
  attaching curves of $\Sigma$ and $\Sigma'$ are dual curves on the
  first stage surface $G_{1\ed 1}$.  Construct a transverse capped
  grope of height $4$ for $\Sigma$ by taking two parallel copies of
  $G^c_{\Sigma'}$, and attaching an annulus cobounded by the boundary
  of the parallel copies, which lies in a regular neighbourhood of
  $\partial\Sigma'$, and which meets $\Sigma$ at exactly one point.
  Note that the caps of the transverse capped grope may meet the caps
  of $G^c_{\Sigma'}$, but no other caps.  Contract the top stage of
  the transverse capped grope, push intersections with caps of
  $G^c_{\Sigma'}$ off the contraction, and then totally contract the
  remaining stages of the transverse capped grope.  This gives us a
  new capped grope, which we still call $G^c_{5}$, together with
  transverse spheres for the stage 2 surfaces.  Each transverse sphere
  meets the new $G^c_{5}$ at a single point.  The transverse spheres
  are not mutually disjoint, but that is permitted.  In particular,
  the two transverse spheres associated to $\Sigma$ and $\Sigma'$,
  which are attached to a dual pair of curves on a stage 1 surface,
  intersect each other in two points. These two points lie in a
  neighbourhood of the intersection point between the pair of curves
  on the stage~1 surface.  Note that while the push off operation
  introduces intersections of caps which are not self-intersections,
  caps of the new $G^c_{5}$ which are attached to distinct top stages
  are still disjoint, since the top stage contraction-push-off can
  only introduce intersections between caps attached to the same top
  stage surface.  This uses Step~2.

  \begin{step-named}[Step 4]
    Extend $G_{1\ed 3} \subset G^c_{5}$ to a properly immersed
    framed height 3 capped grope $G^c_3$ whose caps lie in a regular
    neighbourhood of $G^c_{5}$ and whose double point loops are
    null homotopic in $M\sm G_{1\ed 1}$.
  \end{step-named}

  Consider a stage 4 surface $\Sigma$ of~$G^c_{5}$.  Choose a regular
  neighbourhood $W_\Sigma$ of $G^c_\Sigma$ in the exterior
  of~$G_{1\ed 3}$.  Note that $G^c_\Sigma$ has height~$2$.  The capped
  gropes $G^c_\Sigma$ are mutually disjoint by the penultimate
  sentence of Step 3.  Thus we may assume that the neighbourhoods
  $W_\Sigma$ are disjoint.  Since
  $W_\Sigma\subset \nu G^c_n\sm G_{1\ed 1}$, the image of
  $\pi_1(W_\Sigma)\to \pi_1(M\sm G_{1\ed 1})$ lies in the image of the
  fundamental group of the component of $\nu G^c_n \sm G_{1\ed 1}$
  containing $\Sigma$, which is a good group, by the hypothesis.  By
  Definition~\ref{definition:good-group}, it follows that there is an
  immersed disc in $W_\Sigma$, whose boundary is equal to the base
  circle of $G^c_\Sigma$, and whose double point loops are trivial in
  $\pi_1(M\sm G_{1\ed 1})$.  Replacing each $G^c_\Sigma$ with such an
  immersed disc, we obtain the desired properly immersed framed capped
  grope of height~3.

  \begin{step-named}[Step 5]
    Construct a one storey capped tower and then a flat embedded disc.
  \end{step-named}

  Now consider the union of $G^c_F$ taken over all stage 2 surfaces
  $F$ of~$G^c_3$.  This is a capped grope of height 2, which is
  properly immersed in~$M\sm G_{1\ed 1}$, is $\pi_1$-null, and has
  transverse spheres, by Steps~3 and~4.  It follows from
  Lemma~\ref{lemma:capped-grope-to-one-story-tower} that the body of
  $\bigcup_F G^c_F$ extends to a properly immersed one storey capped
  tower with 3 surface stages in~$M\sm G_{1\ed 1}$.  Attach this one
  storey capped tower to $G_{1\ed 1}$, to obtain a one storey capped
  tower with 4 surface stages.
  Theorem~\ref{theorem:one-story-tower-four-stages-implies-disc} then
  yields a flat embedded disc as claimed.
\end{proof}

\section{Casson towers of height four, three and two}
\label{section:casson-towers}

In this section we apply Theorem~\ref{theorem:grope-to-disc-main} on
gropes to obtain results on Casson towers.  We consider Casson towers
of height four, three and two, in that order.  As the height of the
Casson tower decreases, we need stronger assumptions on fundamental
groups in order to deduce the existence of embedded discs.

The following construction of Ray allows us to pass from Casson towers
to gropes.  Recall that the symplectic basis curves of top stage
surfaces of a framed grope is framed by the induced framing, and
double point loops of a plumbed handle are framed as in
Definition~\ref{definition:casson-tower}.

\begin{proposition}[{\cite[Proposition~3.1]{Ray-2013-1}}]
  \label{proposition:grope-in-casson-tower}
  A Casson tower $T$ of height $n$ contains an embedded framed grope
  $G_n$ of height $n$ with base circle equal to the attaching curve
  $C(T)$ as framed circles.  Moreover the union of the
  standard symplectic basis curves on the top stage surfaces of $G_n$
  is, as a framed 1-submanifold, isotopic to the union of $2^n$
  parallel copies of the double point loops of the top stage of $T$,
  via disjointly embedded framed annuli whose interior is disjoint
  from~$G_n$.
\end{proposition}

We note that the first stage surface $G_{1\ed 1}$ of the grope $G
\subset T$ is denoted by $\Sigma(T)$ in the introduction.

To employ the grope technology, we need capped gropes.  The following
innocent observation is useful in producing capped gropes in a Casson
tower.  We will also present a more involved construction of a capped
grope in
Lemma~\ref{proposition:capped-grope-from-embedded-casson-tower}.

\begin{lemma}[Capped grope in a Casson tower]
  \label{lemma:capped-grope-in-casson-tower}
  A Casson tower $T$ of height $n+1$ contains a properly immersed
  capped grope $G_n^c$ of height $n$, with base circle equal to $C(T)$
  as framed circles.  The body of $G^c_n$ is the grope $G_n$ for the
  subtower $T_{1\ed n}$ from
  Proposition~\ref{proposition:grope-in-casson-tower}.
\end{lemma}

\begin{proof}
  Let $G_n$ be the framed embedded grope in~$T_{1\ed n}$ obtained by
  applying Proposition~\ref{proposition:grope-in-casson-tower}
  to~$T_{1\ed n}$.  We will attach caps to $G_n$, constructed from
  parallel copies of the core discs of the top stage plumbed handles
  of $T$ (together with parallels of the annuli given in
  Proposition~\ref{proposition:grope-in-casson-tower}).  The only
  issue is that the caps should be framed.  For this purpose, we
  arrange that the top stage core discs of $T$ induce the preferred
  framing on the double point loops of the $n$th stage plumbed handles
  to which they are attached.  That is, each core disc of a stage
  $n+1$ plumbed handle should have the signed count of its
  self-plumbings equal to zero.  We achieve this by locally
  introducing the requisite number of self-plumbings of appropriate
  sign.  Now, from this and from the framing property in
  Proposition~\ref{proposition:grope-in-casson-tower}, it follows that
  we obtain framed caps.  Thus we have a properly immersed framed
  capped grope $G_n^c$ extending~$G_n$, inside~$T$.
\end{proof}

We also need the following lemma on a fundamental group arising from
the construction in
Proposition~\ref{proposition:grope-in-casson-tower}.

\begin{lemma}\label{lemma:pi-1-T_1-1-setminus-G_1-1}
  Let $T$ be a Casson tower and let $\Sigma=\Sigma(T)$.  Then
  \[
  \pi_1(T_{1\ed 1} \sm \Sigma) \cong \la \mu,a_1,\dots,a_k \, | \,
  [\mu,a_i],\, i=1,\dots, k \ra
  \]
  where $\mu$ is a meridian to $\Sigma$ and the $a_i$ are the double
  point loops of $T_{1\ed 1}$.
\end{lemma}

\begin{proof}
  For convenience, we assume that the first stage of $T$ has one
  (negative) double point.  We will indicate along the way how to
  adapt the proof for the general case.

  We recall Ray's construction from \cite[Proof of
  Proposition~3.1]{Ray-2013-1}.  The first stage surface
  $\Sigma$ is shown in Figure~\ref{figure:ray-surface}, where the
  plumbed handle $T_{1\ed1}$ is described as a Kirby diagram.  More
  precisely, $\Sigma$ is obtained by pushing the interior of the
  surface in Figure~\ref{figure:ray-surface} slightly into the
  interior of~$T_{1\ed1}$.  The curve $a_1$ is the double point loop,
  which is the attaching circle for the next stage of the Casson tower
  $T$.

  \begin{figure}[ht]
    \labellist\small
    \pinlabel{$C(T)$} at 0 53
    \pinlabel{$\Sigma$} at 168 140
    \pinlabel{$a_1$} at 170 20
    \endlabellist
    \includegraphics[scale=.7]{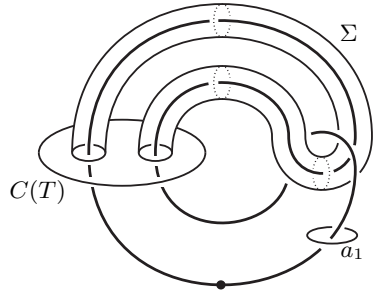}
    \caption{Kirby diagram of a plumbed handle together with Ray's genus
      one surface.}
    \label{figure:ray-surface}
  \end{figure}

  We remark that the commutator relation $[\mu,a_1]$ can be seen as
  follows: the normal circle bundle of $\Sigma$ restricted on one of
  the dotted circles in Figure~\ref{figure:ray-surface} is a torus
  disjoint from $\Sigma$, whose symplectic basis curves are (isotopic
  to) a meridian of $\Sigma$ and the curve~$a_1$.  In what follows we
  will prove that this commutator relation suffices to present
  $\pi_1(T_{1\ed 1} \sm \Sigma)$.

  Consider a collar neighbourhood $\partial T_{1\ed1} \times I$ of
  $\partial T_{1\ed1}$.  We may assume that the height function
  $\partial T_{1\ed1} \times I \to I$ restricts to a Morse function
  for $\Sigma$ with 3 critical points, corresponding to two 1-handles
  and one 2-handle of~$\Sigma$; the 1-handles are shown as dashed
  lines in Figure~\ref{figure:surface-handle-decomp}.

  \begin{figure}[ht]
    \includegraphics[scale=.9]{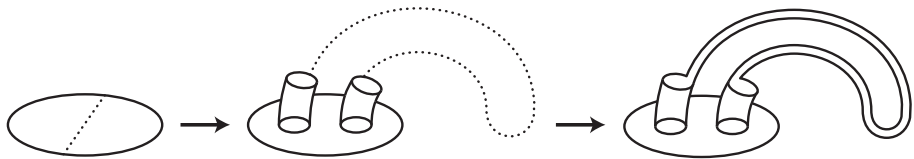}
    \caption{A handle decomposition of the surface~$\Sigma$.}
    \label{figure:surface-handle-decomp}
  \end{figure}

  The handle decomposition of $\Sigma$ gives rise to a handle
  decomposition of the exterior of~$\Sigma$ in the collar
  neighbourhood $\partial T_{1\ed1}\times I$ of $\partial T_{1\ed1}$:
  \[
  (\partial T_{1\ed1}\times I) \sm \Sigma = (\partial T_{1\ed1}
  \sm \partial \Sigma) \times I \cup (\text{two 2-handles}) \cup
  (\text{one 3-handle}).
  \]
  An $i$-handle in a handle decomposition of a surface embedded in a
  4-manifold corresponds to an $(i+1)$-handle in a handle
  decomposition of the exterior of the surface (for a proof, see for
  example~\cite[Proposition~6.2.1]{Gompf-stipsicz-book}).  The handle
  attachments are shown in Figure~\ref{figure:exterior-handle-decomp},
  where the attaching circles and spheres are drawn with dashed lines.  The dashed lines showing the 1-handle attachments in the left and middle diagrams of Figure~\ref{figure:surface-handle-decomp} correspond to the 2-handle attachments in the construction of the exterior of the surface shown in the top left and top centre diagrams of Figure~\ref{figure:exterior-handle-decomp}.

  \begin{figure}[ht]
    \labellist\small
    \pinlabel{$x$} at -5 163
    \pinlabel{$y$} at 32 180
    \pinlabel{$p$} at 65 217
    \pinlabel{$s$} at 60 179
    \pinlabel{$r$} at 60 148
    \pinlabel{$q$} at 65 110
    \endlabellist
    \includegraphics[scale=.85]{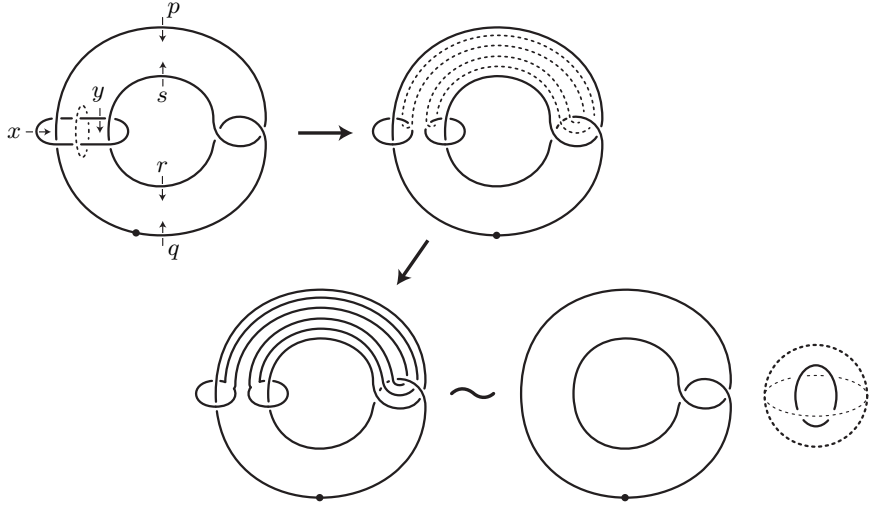}
    \caption{A handle decomposition of $(\partial T_{1\ed1}\times I)
      \sm \Sigma$.  An arrow indicates the attaching of a handle,
      and the $\sim$ symbol indicates an isotopy.  It is much easier
      to draw the attaching 2-sphere for the 3-handle after the
      isotopy.}
    \label{figure:exterior-handle-decomp}
  \end{figure}

  As shown in Figure~\ref{figure:ray-surface}, $\partial
  T_{1\ed1}\sm \partial \Sigma$ is the exterior of the Whitehead
  link, with a boundary component (corresponding to the dotted circle)
  filled in with a solid torus along the zero framing.  Therefore,
  starting with the Wirtinger presentation
  \[
  \langle x, y, p, q, r, s \mid y=p^{-1}xp,\, q = xpx^{-1},\, r=sqs^{-1},\,
  s=q^{-1}pq,\, r = xsx^{-1} \rangle,
  \]
  of the Whitehead link, where the generators are those shown in the
  first diagram in Figure~\ref{figure:exterior-handle-decomp}, and
  then adding three more relators
  \[
  x^{-1}s^{-1}xq^{-1}p^2,\, xy^{-1},\, xqx^{-1}q^{-1}
  \]
  which are from the Dehn filling and the 2-handle attachments
  respectively, we obtain a presentation of $\pi_1((\partial
  T_{1\ed1}\times I) \sm \Sigma)$.  Simplifying the presentation,
  we obtain:
  \[
  \pi_1((\partial T_{1\ed1}\times I) \sm \Sigma) \cong \langle x, q \mid
  [x,q] \rangle.
  \]
  Observe that $x=\mu$ and $q=a_1$.  In the case that the clasp is a
  positive clasp, the relators $r=sqs^{-1}$ and $s=p^{-1}qp$ are
  replaced by $s=rpr^{-1}$ and $r=p^{-1}qp$.  It is not too hard to
  check that the computation above has the same outcome with these
  alterations.  The relator corresponding to Dehn filling also changes
  to $x^{-1}pxrp^{-2}$, but this relator is superfluous to simplifying
  the presentation in both cases.

  Turn the handle decomposition of $T_{1\ed1}$, into a 0-handle and a
  1-handle, given by the Kirby diagram in
  Figure~\ref{figure:ray-surface} upside down. We see that
  $T_{1\ed1}\sm \Sigma$ is obtained by attaching a 3-handle and a
  4-handle to $(\partial T_{1\ed1}\times I) \sm \Sigma$.  In general,
  turning a handle decomposition
  \[
  \bigcup h_0 \cup \dots\cup \bigcup h_3
  \]
  of a connected $4$-manifold with nonempty boundary $(M,\partial M)$
  upside down gives us a decomposition rel.\ boundary
  \[
  \partial M \times I \cup \bigcup h_3^* \cup \dots\cup \bigcup h_0^*,
  \]
  where $h_i^*$ is the $(4-i)$-handle dual to the $i$-handle~$h_i$.
  Since neither a 3-handle nor a 4-handle affect the fundamental
  group,
  \[
  \pi_1(T_{1\ed1} \sm \Sigma) \cong \pi_1((\partial
  T_{1\ed1}\times I) \sm \Sigma) \cong \langle \mu, a_1 \mid
  [\mu,a_1] \rangle.
  \]
  For $k>1$ double points, take $k$ copies of
  Figure~\ref{figure:ray-surface}, with the crossings in the clasp
  switched where appropriate, and connect sum the $C(T)$ curves
  together.  This performs a boundary connect sum operation on the
  surfaces.  Similarly, take multiple copies of the first three
  diagrams of Figure~\ref{figure:exterior-handle-decomp}, and connect
  sum the $C(T)$ curves together i.e.\ the copies of the curve with
  meridians $x$ and $y$ in the first diagram of
  Figure~\ref{figure:exterior-handle-decomp}.  This composite curve
  represents $\partial \Sigma$.  There is still only one 2-handle of
  $\Sigma$, therefore only one 3-handle of $(\partial T_{1\ed 1}\times
  I)\sm \Sigma$.  So in the ramified case the analogue of the final
  diagram of Figure~\ref{figure:exterior-handle-decomp} will still
  have just a single dashed 2-sphere.  The Seifert-Van Kampen theorem
  applies to show that the effect of the connect sum operations is to
  identify the meridians labelled $x$ in all the copies of the
  diagrams from Figure~\ref{figure:exterior-handle-decomp}; all of
  these become the meridian $\mu$.  Indeed this is the only effect. By
  the computation above, $\mu$ commutes with all the double point
  loops.  We therefore have the presentation
  \[
  \pi_1(T_{1\ed 1} \sm \Sigma) \cong \la \mu,a_1,\dots,a_k \, |
  \, [\mu,a_i],\, i=1,\dots, k \ra. \qedhere
  \]
\end{proof}

\subsection{Casson towers of height four and three}
\label{section:casson-towers-height-four-three}

In this subsection we will prove Theorems~\ref{theorem:main}
and~\ref{theorem:main-height-3} from the introduction.

\begin{definition}\label{definition:distorted-tower}
  Take a Casson tower of height~$n$ and introduce any number of
  plumbings between any top stage handle and any other handle in stage
  two or higher.  A $4$-manifold with boundary, together with a framed
  embedded circle $C$ in its boundary, that is obtained in this way,
  is called a \emph{distorted Casson tower of height $n$}.
\end{definition}

Note that a Casson tower of height $n$ is a distorted Casson tower of
height~$n$.  As another example, a distorted tower of height $n$ may
arise if we have a height~$n-1$ tower $T$ embedded in a 4-manifold
$(M,\partial M)$ with $\partial_- \subset \partial M$, and the double
point loops of the top stage are null-homotopic in the complement in
$M$ of the first stage $T_{1\ed 1}$ of~$T$.  Then a neighbourhood of
the union of the height~$n-1$ tower and the null-homotopies of the
double point loops gives rise to a distorted height~$n$ Casson tower.
Some care is needed to frame the null-homotopies.  Null-homotopies of
different double point loops may intersect each other, and stages two
or higher of the height~$n-1$ Casson tower.

Theorem~\ref{theorem:main} says that a distorted Casson tower $T$ of height 4
contains a disc bounded by~$C(T)$.
We note that we do not make any
assumptions about embedding a distorted Casson tower of height 4; the
distorted Casson tower itself is considered as the ambient manifold.

\begin{proof}[Proof of Theorem~\ref{theorem:main}]
  Let $T$ be a distorted Casson tower of height~4.  First, we apply
  Lemma~\ref{lemma:capped-grope-in-casson-tower} to $T_{1\ed 3}$ to
  obtain a properly immersed capped grope $G^c_2$ in $T_{1\ed 3}$,
  which is bounded by the framed circle~$C(T)$.

  Recall that a plumbed handle is diffeomorphic to a 4-ball with
  1-handles attached, and that the fundamental group is generated by
  the double point loops.  By induction, the
  fundamental group of a Casson tower is generated by the top stage
  double point loops.  Applying this to our case, we see that the
  inclusion induced map $\pi_1(T_{2\ed 3}) \to \pi_1(T)$ is trivial,
  since the 4th stage discs give null-homotopies for the double point
  loops of the 3rd stage plumbed handles.  By
  Lemma~\ref{lemma:pi-1-T_1-1-setminus-G_1-1} and a straightforward
  Seifert-Van Kampen theorem computation for $T_{1\ed 3} \sm G_{1\ed
    1} = (T_{1\ed 1} \sm G_{1\ed 1})\cup T_{2\ed 3}$, it follows that
  the image of $\pi_1(T_{1\ed 3} \sm G_{1\ed 1})$ in $\pi_1(T\sm
  G_{1\ed 1})$ under the inclusion induced map is isomorphic to $\Z$,
  generated by a meridian of~$G_{1\ed 1}$.  From this it also follows
  that the image of $\pi_1(\nu G^c_2 \sm G_{1\ed 1})$ in $\pi_1(T\sm
  G_{1\ed 1})$ is the same~$\Z$.  An infinite cyclic group has
  subexponential growth and is therefore good by
  Theorem~\ref{theorem:krushkal-quinn}, so the hypothesis of
  Theorem~\ref{theorem:grope-to-disc-main} is satisfied (in our case,
  $G_2^c$ is connected and therefore so is $\nu G^c_2 \sm G_{1\ed 1}$).
  Applying Theorem~\ref{theorem:grope-to-disc-main} we can find a
  flat embedded disc inside $T$ as claimed.
\end{proof}

In order for a Casson tower of height 3 to suffice for the existence
of an embedded disc, we will need to embed the tower into a
4-manifold, with a fairly strong assumption on fundamental groups.
Recall that for a Casson tower $T$, there is a surface $\Sigma(T)$
contained in the first stage $T_{1\ed 1}$, by
Proposition~\ref{proposition:grope-in-casson-tower} (see also
Figure~\ref{figure:ray-surface}).  Theorem~\ref{theorem:main-height-3}
says the following: \emph{Let $W$ be a $4$-manifold with boundary and
  suppose that $T=\bigsqcup T_i$ is a collection of disjoint Casson
  towers $T_i$ of height~$3$ in $W$ such that
  $\partial_-(T)\subset \partial W$ and the image of
  $\pi_1(T_i \sm \Sigma(T_i)) \to \pi_1(W \sm\Sigma(T))$ is a good
  group for each~$i$.  Then the framed link $C(T) \subset \partial W$
  is slice in~$W$.}

We note that even if $W$ is simply connected, it is quite possible
that the image of this fundamental group in $\pi_1(W\sm G_{1\ed 1})$
will not satisfy the good property hypothesis. For example if $T$ has
more than one component, $\pi_1(W\sm G_{1\ed 1})$ might contain a
non-abelian free group generated by meridians to the connected
components of $G_{1\ed 1}$.

\begin{proof}[Proof of Theorem~\ref{theorem:main-height-3}]

  The proof begins the same as the proof of
  Theorem~\ref{theorem:main}\@.  Let $G_2^c$ be the height 2 capped
  grope in $T$, constructed in the same way as in that proof.  Note that now $G_2^c$
  may not be connected.  A component of $\nu G_2^c$, say $V$, lies in
  some component $T_i$ of~$T$.  The inclusion induced
  homomorphisms on fundamental groups factor as
  \[
  \pi_1(V \sm G_{1\ed 1}, *) \to \pi_1(T_i \sm G_{1\ed 1}, *) \to
  \pi_1(W \sm G_{1\ed 1}, *)
  \]
  for any choice of basepoint $*$ in~$V$.  The hypothesis of the
  theorem, that the image of the second map is good, implies that the
  image of this composite homomorphism is also good.  This holds
  because the good property is closed under taking subgroups.  Thus
  we can apply Theorem~\ref{theorem:grope-to-disc-main} to obtain the
  flat embedded discs that we seek.
\end{proof}

We remark that, as seen from the above proofs,
Theorem~\ref{theorem:main} is indeed a consequence of
Theorem~\ref{theorem:main-height-3}.  We point out that so far we only
used the Disc Embedding Theorem~\ref{theorem:grope-to-disc-main} for a
height 2 capped group, although it holds for any height $\ge 1.5$.

\subsection{Casson towers of height two}
\label{section:casson-towers-height-two}

This section contains our strongest conclusions in terms of the height
of Casson towers, using the strongest assumption, namely triviality,
on fundamental groups.  Height 2 Casson towers seem to be the most
useful for slicing knots and links in $D^4$, since in practice,
contrivances notwithstanding, it is often difficult to construct tall
Casson towers.  Applications will be given in
Section~\ref{section:slice-knots}.

Our arguments for height 2 Casson towers also involve capped gropes.
For this case, we need the full power, in terms of height, of
Theorem~\ref{theorem:grope-to-disc-main}. That is, we apply the
theorem to a height 1.5 capped grope.  Also, we need a new
construction of a capped grope from an \emph{embedded} Casson tower,
which is given below.  The construction relies upon the properties of
the embedding of the tower, and will not work without an ambient
manifold.

\begin{proposition}[Capped grope from an embedded Casson tower]
  \label{proposition:capped-grope-from-embedded-casson-tower}
    Suppose $n>0$, $(T,\partial_-)$ is a
  Casson tower of height $n+1$ embedded in a 4-manifold $(M,\partial
  M)$, and the double point loops of the top stage of $T$ are
  null-homotopic in $M \sm T_{1\ed n}$.  Then there is a properly
  immersed capped grope $G_{n.5}^c$ of height $n.5$ in $M$, which
  extends the grope $G_n$ for the subtower $T_{1\ed n}$ from
  Proposition~\ref{proposition:grope-in-casson-tower}.  In particular,
  the first stage surface of $G_{n.5}^c$ is $\Sigma(T)$, and the
  attaching circle of $G_{n.5}^c$ and $C(T)$ are equal as framed
  circles.
\end{proposition}

In fact, we only need the $n=1$ case of
Proposition~\ref{proposition:capped-grope-from-embedded-casson-tower}
in this paper, but we state and prove it for general $n>0$ for
possible later use, since this does not require any additional
complication.

For the proof of Proposition
\ref{proposition:capped-grope-from-embedded-casson-tower}, we begin
with a couple of lemmata.  The following lemma cleans up a certain
type of naturally occurring capped grope to a properly immersed
height~2 framed capped grope.  We need to use the notion of a twisted
cap, defined as follows. Suppose $G$ is a framed grope embedded in a
4-manifold $W$.  An immersed disc $D$ in $W$ bounded by a symplectic
basis curve of a top stage surface $\Sigma$ of $G$ is called a
\emph{$\pm1$ twisted cap} if the interior of a collar neighbourhood of
$\partial D \subset D$ is disjoint from $G$, and a push-off of
$\partial D$ along $\Sigma$ induces a section of the normal bundle of
$D$ with relative Euler number $\pm1$. That is, the push-off gives
rise to the framing on $\partial D$ obtained by twisting the
restriction of the unique framing on~$D$ once (either positively or
negatively).

\begin{lemma}\label{lemma:big-small-caps}
  Suppose we are given a height~2 capped grope~$G^c$ which is immersed
  (not properly) in a 4-manifold~$M$, satisfying the following.  The
  surface stages are disjointly embedded and framed.  Each dual pair
  of curves on a second stage surface has two caps, one $\pm
  1$~twisted cap which is disjoint from the body of the grope, and one
  framed cap which potentially intersects other caps \emph{and} second
  stage surfaces.  Then there is a properly immersed height~$1.5$
  framed capped grope in a neighbourhood of $G^c$ with the same first
  stage surface.
\end{lemma}

\begin{proof}
  Divide the second stage surfaces and caps into two sides, the $+$
  and $-$ sides, as described just prior to
  Lemma~\ref{lemma:cap-separation}.
  There are two problems to be dealt with, namely the twisted caps
  need to be framed and their dual caps need to be made disjoint from
  the body of the grope.  We have the freedom to reduce height by one
  on the $-$ side.  We will modify the $-$ side first and then improve
  the $+$ side.

  Call the $\pm 1$ twisted caps the \emph{small} caps, and the other
  caps, which can intersect second stage grope surfaces as well as
  each other, the \emph{big} caps.

  Forget the $-$ side big caps and apply the boundary twisting
  operation \cite[Section~1.3]{Freedman-Quinn:1990-1} to the $-$ side
  small caps, so that they are framed with respect to the $-$ side
  surfaces.  This introduces an intersection of each $-$ side small
  cap with the $-$ side surface to which it is attached.  Then use the $-$ side
  small caps to perform \emph{asymmetric surgery} on the $-$ side
  surfaces, changing them to immersed discs, which are the new $-$
  caps.  The intersection of the small caps with the second stage
  surface introduced by the boundary twisting gives rise to
  self-intersections of the new $-$ caps.  These new $-$ caps may also
  intersect other caps, but that is permitted.

  Now create transverse spheres to the $+$ side surface stages using
  two parallel copies of the new $-$ side caps, and the annulus in the
  normal circle bundle to the attaching circle of the $-$ cap, as we
  have done in several other instances in this paper (see the first
  paragraph of the proof of the Cap Separation
  Lemma~\ref{lemma:cap-separation} and Step 3 of the proof of
  Theorem~\ref{theorem:grope-to-disc-main}).  The transverse spheres
  we construct are immersed and may intersect $+$ and $-$ side caps.

  Boundary twist the $+$ side small caps to frame them with respect to
  the $+$ side surfaces.  This creates intersections of the $+$ side
  small caps with the $+$ side second stage surfaces.  Now we just
  have to remove intersections of the $+$ side caps, both big and
  small, with the $+$ side surfaces.  To achieve this, tube anything
  that intersects a $+$ side surface into a transverse sphere.  We
  obtain a properly immersed height~$1.5$ framed capped grope as
  required.
\end{proof}

\begin{lemma}
  \label{lemma:resolving-double-points}
  In a plumbed handle $T$ with $k$ plumbings, there is a genus $k$
  framed surface $F$ with $\partial F=C(T)$ as framed circles, which
  has symplectic basis curves
  $\alpha_1,\beta_1,\ldots,\alpha_k,\beta_k$ satisfying the following.
  Each $\alpha_i$ is dual to $\beta_i$, each $\alpha_i$ bounds a
  $\pm1$ twisted cap whose interior is disjoint from $F$. The union
  $\bigsqcup_i \beta_i$ is parallel to the union of the double point
  loops of $T$ on $\partial_+T$, via $k$ framed annuli disjointly
  embedded in $T$, whose interior is disjoint from $F$ and from the
  $\pm1$ twisted caps.
\end{lemma}

\begin{proof}
  The surface $F$ is obtained from the core disc of $T$ by a standard
  construction that resolves singularities by increasing the genus:
  replace the standard local model $(D^2\times D^2, D^2\times 0 \cup
  0\times D^2)$ of an intersection point by a twisted annulus in $S^3
  = \partial(D^2\times D^2) \subset D^2\times D^2$ bounded by the Hopf
  link $S^1\times 0 \cup 0\times S^1$.  To verify the framing
  assertions, we use a Kirby diagram argument as follows.

  The Kirby diagram of a plumbed handle with $k$ plumbings in
  Figure~\ref{fig:plumbed-handle-diagram} is isotopic to the diagram
  in Figure~\ref{figure:plumbed-handle-diagram-2}.  Observe that
  $C(T)$ bounds a surface $F$ which is a band sum of $k$ untwisted
  annuli with $k$ twisted 1-handles attached; see
  Figure~\ref{figure:plumbed-handle-diagram-2}.  Let $\alpha_i$ be the
  core circle of the $i$th twisted 1-handle, and $\beta_i$ be the core
  of the $i$th untwisted annulus; see
  Figure~\ref{figure:plumbed-handle-diagram-2} again.  The curves
  $\alpha_i$ and $\beta_i$ form a symplectic basis.

  \begin{figure}[t]
    \labellist
    \small
    \pinlabel {$C(T)$} at 37 7
    \pinlabel {$a_1$} at 38 110
    \pinlabel {$a_k$} at 167 110
    \pinlabel {$F$} at 132 20
    \pinlabel {$\alpha_1$} at -5 38
    \pinlabel {$\alpha_k$} at 124 38
    \pinlabel {$\beta_1$} at 100 30
    \pinlabel {$\beta_k$} at 230 30
    \endlabellist
    \includegraphics{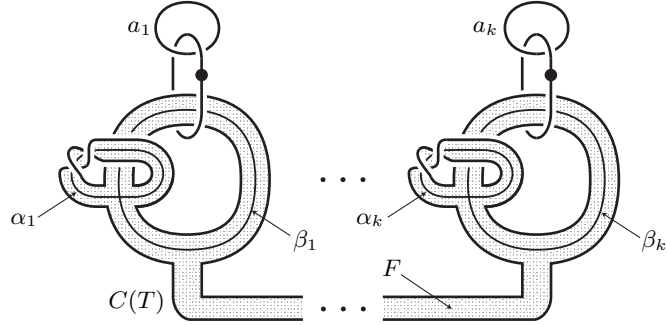}
    \caption{A Kirby diagram of a plumbed handle $T$ with $k$ plumbings,
      and a genus $k$ surface $F$ bounded by $C(T)$.}
    \label{figure:plumbed-handle-diagram-2}
  \end{figure}

  From Figure~\ref{figure:plumbed-handle-diagram-2}, we see that $F$
  induces the zero framing on the $\beta_i$, and $\beta_i$ is parallel
  to the double point loop $a_i$ via a framed annulus inducing the
  0-framing on both, as desired.  Also, $\alpha_i$ bounds a $\pm1$
  twisted cap whose interior lies in the interior of $D^4$ and so is
  disjoint from everything else.  Since $F$ is a Seifert surface for
  $C(T)$, $F$ induces the zero framing on $C(T)$, which is the
  preferred framing by Definition~\ref{definition:casson-tower}.
\end{proof}

\begin{proof}[Proof of
  Proposition~\ref{proposition:capped-grope-from-embedded-casson-tower}]

  First apply Proposition~\ref{proposition:grope-in-casson-tower}
  to find a framed grope $G_n$ in $T_{1\ed n}$ with framed
  boundary~$C(T)$.  Consider the top stage of $T$ in
  $V:=\overline{M-T_{1\ed n}}$.  By
  Lemma~\ref{lemma:resolving-double-points}, in each stage $n+1$
  plumbed handle of $T$, we have a framed surface bounded by its
  attaching circle, with $\pm1$ twisted caps which we call a
  \emph{small cap}, and annuli cobounded by the dual basis curve and
  the double point loop.  By hypothesis, the double point loop is
  null-homotopic in~$V$.  Attach a null-homotopy to each annulus to
  obtain a cap dual to the $\pm1$ twisted cap; we call it a \emph{big
    cap}.  This terminology was already used in the proof of
  Lemma~\ref{lemma:big-small-caps}.  We may assume that each big cap
  is framed, by applying boundary twist if necessary.  This gives a
  capped grope of height 1, which is immersed in $V$ but not properly
  immersed in general; the big caps may intersect other surfaces and
  caps.

  Take $2^n$ push-offs of each of these height 1 capped grope and
  attach them to $G_n$ to obtain a height $n+1$ capped grope
  $G_{n+1}^c$.  The body of $G_{n+1}^c$ and the big caps are
  compatibly framed, while the small caps are twisted.  Now the big
  caps may intersect stage $n+1$ surfaces and other caps of
  $G_{n+1}^c$, but are disjoint from $G_{1\ed n}$, since the big caps
  lie in~$V$.  Note that now there are intersections between the small
  caps, which were introduced when we took push-offs, since the small
  caps are twisted.  However the small caps are disjoint from the body
  of~$G_{n+1}^c$.

  Consider $G_{n\ed (n+1)}^c$, the top two surface stages of
  $G_{n+1}^c$ together with all the small and big caps.  Since
  $G_{n\ed (n+1)}^c$ is a height $2$ capped grope satisfying the
  hypotheses of Lemma~\ref{lemma:big-small-caps}, there is a properly
  immersed height 1.5 capped grope with the same base surfaces, which
  lies in a neighbourhood of~$G_{n\ed(n+1)}^c$.  Replace $G_{n\ed
    (n+1)}^c$ in $G_{n+1}^c$ with this height 1.5 capped grope to
  obtain a properly immersed capped grope~$G_{n.5}^c$ of height~$n.5$
  as desired.
\end{proof}

We are now ready to give the proof of
Theorem~\ref{theorem:main-height-2} from the introduction, which says:
\emph{let $W$ be a $4$-manifold with boundary and suppose $T$ is a
  Casson tower of height~$2$ embedded in $W$ such that the second
  stage $T_{2\ed 2}$ of $T$ lies in a codimension zero simply
  connected submanifold $V \subseteq \overline{W \sm
    T_{1\ed 1}}$.  Then the knot $C(T)\subset \partial W$ is
  slice in~$W$.}

\begin{proof}[Proof of Theorem~\ref{theorem:main-height-2}]
  Apply Proposition~\ref{proposition:capped-grope-from-embedded-casson-tower}
  to the given Casson tower $T$ of height~2, with $M:=T\cup V$, to
  obtain a properly immersed capped grope~$G_{1.5}^c$ in~$T\cup V$.

  Observe that
  $(T\cup V) \sm G_{1\ed 1} = V\cup (T_{1{\ed 1}} \sm G_{1\ed 1})$,
  where $V$ and $(T_{1{\ed 1}} \sm G_{1\ed 1})$ are glued along
  neighbourhoods of the attaching curves for $T_{2\ed 2}$.  By
  Lemma~\ref{lemma:pi-1-T_1-1-setminus-G_1-1} and a straightforward
  application of the Seifert-Van Kampen theorem, it follows that
  $\pi_1(V\cup (T_{1{\ed 1}} \sm G_{1\ed 1})) \cong \Z$, which is good
  by Theorem~\ref{theorem:krushkal-quinn}.  Apply
  Theorem~\ref{theorem:grope-to-disc-main} to $G_{1.5}^c$ in
  $M:= T\cup V$, to yield a flat embedded disc in $T\cup V$ bounded
  by~$C(T)$.
\end{proof}

\begin{remark}
  In the above proof of Theorem~\ref{theorem:main-height-2}, we have
  shown that $C(T)$ is slice in the submanifold $T\cup V\subset W$.
\end{remark}

\section{Slice knots}
\label{section:slice-knots}

In this section we apply the results on Casson towers of
Section~\ref{section:casson-towers} to produce a new family of slice
knots in~$S^3$.

\subsection{Band sums of Whitehead doubles}
\label{subsection:band-sum-whitehead-double}

In this subsection, as promised in the introduction, we use
Theorem~\ref{theorem:main-height-2} to give a proof of
Theorem~\ref{theorem:main-slice-knots}, which we state here again for
the reader's convenience: \emph{suppose $L$ is an $m$-component
  homotopically trivial link, and $K$ is a knot obtained from $\Wh(L)$
  by applying $m-1$ band sum operations.  Then $K$ is slice.}


We begin, in Lemma~\ref{lemma:Wh-plumbed-handles}, with a well-known
observation on Whitehead doubles and plumbed handles.  To state it we
recall the definition of the (untwisted) Whitehead double of a framed
link in a general 3-manifold.  Let $\Wh\subset S^1\times D^2$ be the
standard untwisted Whitehead knot, that is, it is obtained by taking
the exterior of a component of a Whitehead link and then identifying
it with $S^1\times D^2$ under the zero framing.  (There are two
possibilities, $\Wh_+$ and $\Wh_-$, depending on the sign of the
clasp.)  The zero framing on the Whitehead link induces a framing on
$\Wh\subset S^1\times D^2$ which we call the zero framing.  For a
framed link $L$ in a general 3-manifold $M$, form an untwisted
Whitehead double $\Wh(L)$ of~$L$, which is a framed link, by replacing
a tubular neighbourhood of each component of $L$ with $(S^1\times
D^2,\Wh)$ under the framing of~$L$.  We also recall that the attaching
circle and double point loops of a plumbed handle are framed as in
Definition~\ref{definition:casson-tower}.

\begin{lemma}[Plumbed handles for Whitehead doubles]
  \label{lemma:Wh-plumbed-handles}
  Suppose $L$ is an $m$-component framed link in a 3-manifold~$M$.
  Then there exist plumbed handles $T_i$
  \textup{(}$i=1,\ldots,m$\textup{)} disjointly embedded in
  $M\times[0,1]$ such that each $T_i$ has exactly one self-plumbing
  with double point loop $\alpha_i$, $\bigsqcup_i C(T_i) =
  \Wh(L)\times 0$, $\bigsqcup_i \alpha_i = L\times 1$ as framed links,
  $T_i\cap (M\times 0) = \partial_-(T_i)$, and $T_i\cap (M\times 1)$
  is a tubular neighbourhood of $\alpha_i \subset \partial_+(T_i)$.
\end{lemma}

We call the plumbed handles in Lemma~\ref{lemma:Wh-plumbed-handles} the
\emph{standard plumbed handles} between $\Wh(L)$ and~$L$.

\begin{proof}
  The best geometric way to understand our plumbed handles $T_i$ is to
  construct the core discs directly: undo the clasp of the
  Whitehead doubling operations on $L$ via a regular homotopy, and
  cap off the resulting trivial link with disjoint discs. We
  obtain an immersion of $m$~discs in $M\times [0,1]$ bounded by
  $\Wh(L)\times 0$, and then by thickening this, we obtain the plumbed
  handles~$T_i$.  Furthermore, in this construction, by regarding the
  regular isotopy as a movie picture of the core discs, it can be seen
  that the double point loops on the core discs can be pushed to
  $L\times 1$ along embedded annuli.  It follows that we may thicken
  the core discs in such a way that $T_i\cap(M\times 1)$ is a tubular
  neighbourhood of the double point loop
  $\alpha_i\subset \partial_+(T_i)$.  The framing condition can also
  be verified by investigating the movie picture carefully.

  The above assertions can be verified rigorously by the following
  alternative description.  Recall that a plumbed handle $T$ with one
  self-plumbing, together with the attaching circle $C$ and the double
  point loop $a_1$, is described by the standard Kirby diagram in
  Figure~\ref{fig:plumbed-handle-diagram} ($k=1$ for now), where $C$
  and $a_1$ are zero framed by
  Definition~\ref{definition:casson-tower}.  In particular
  $T\cong S^1\times D^3$.  By straightening the dotted circle in the
  Kirby diagram, it follows that if we write
  $T=S^1\times D^2\times I$, then we may assume
  $C = \Wh\times 0 \subset S^1\times D^2\times 0$ as framed circles,
  and $a_1 = S^1\times 0 \times 1 \subset S^1\times D^2\times 1$,
  framed by the product structure.  Now, the framing of $L$ gives us
  an identification of a tubular neighbourhood of
  $L\times [0,1]\subset M\times[0,1]$ with
  $L\times D^2\times I = \bigsqcup^m (S^1\times D^2\times I) =
  \bigsqcup^m T$,
  $m$ disjoint plumbed handles.  By the above and by the definition of
  $\Wh(L)$, it follows that the attaching circles of these plumbed
  handles form the framed link $\Wh(L)\times 0$, and the double point
  loops form the framed link $L\times 1$.
\end{proof}

\begin{proof}[Proof of Theorem~\ref{theorem:main-slice-knots}]

  Attach the bands used in the band sum operations for
  $\Wh(L)=\Wh(L)\times 0$ to the annuli
  $\Wh(L)\times[0,\frac12]\subset S^3\times[0,\frac12]$ and push them
  slightly, to obtain a planar surface with $m+1$ boundary components in $S^3\times [0,\frac12]$
  cobounded by $K\times 0$ and~$\Wh(L)\times\frac12$.  The zero
  framings on $\Wh(L)$ and $K$ extend to a framing of the planar
  surface.  Thicken the planar surface in $S^3\times[0,\frac12]$ and
  attach the standard plumbed handles in $S^3\times[\frac12,1]$
  between $\Wh(L)$ and $L$ given by
  Lemma~\ref{lemma:Wh-plumbed-handles}. This constructs a single
  plumbed handle $T_{1\ed1}$ embedded in $S^3\times[0,1]$.  It has
  $K\times 0$ as the attaching circle and $L\times 1$ as the double
  points, by Lemma~\ref{lemma:Wh-plumbed-handles}.  We will use
  $T_{1\ed1}$ as the first stage of a Casson tower.

  Next, view $S^3\times [0,1]$ as a collar neighbourhood of the
  boundary of $D^4 = S^3\times [0,1] \cup_{S^3\times 1} (\text{smaller
  } D^4)$.  Since $L$ is homotopically trivial, there are disjoint
  immersed discs in the smaller $D^4$, which can be thickened to
  plumbed handles whose attaching circles form $L\times 1$.  We want
  to use these plumbed handles for the second stage.  Observe the
  following general fact, which follows from
  Definition~\ref{definition:casson-tower}: the preferred framing of
  the attaching circle of a plumbed handle embedded in $D^4$ with
  $\partial_-\subset S^3$ is the zero framing in~$S^3$.  Apply this to
  our case: since the double point loops $L\times 1$ of~$T_{1\ed 1}$
  are zero framed by Lemma~\ref{lemma:Wh-plumbed-handles}, it follows
  that we can attach these plumbed handles to $T_{1\ed1}$, in the
  smaller $D^4$, to yield a height 2 Casson tower, say~$T$.

  By construction, $C(T)=K$, the second stage $T_{2\ed2}$ of $T$
  lies in the smaller~$D^4$, and the first stage $T_{1\ed 1}$ lies in
  the collar $S^3\times[0,1]$ of the boundary of the bigger
  4-ball~$D^4$.  Since the smaller $D^4$ is simply connected, we can
  apply Theorem~\ref{theorem:main-height-2} to obtain a flat
  embedded disc bounded by $K$ as claimed.
\end{proof}

The following is an immediate corollary.  Recall that a link $L$ in
$S^3$ is \emph{weakly slice} if it bounds a flat embedding of a planar
surface in~$D^4$.

\begin{corollary}\label{cor:wh-K-slice}
  The Whitehead double of any homotopically trivial link is weakly
  slice.
\end{corollary}

\begin{proof}
  If $L$ is an $m$-component homotopically trivial link, then a knot
  $K$ obtained by $m-1$ bands sum operations on $\Wh(L)$ is slice by
  Theorem~\ref{theorem:main-slice-knots}.  Attaching the bands to a
  slicing disc and pushing them slightly, we obtain a punctured disc
  bounded by $\Wh(L)$.
\end{proof}

As another consequence of Theorem~\ref{theorem:main-slice-knots}, we
prove Corollary~\ref{corollary:main-whitehead-ribbon-slice}, which
says the following: \emph{suppose $L$ is an $m$-component
  homotopically trivial link, and~$R$ is a ribbon knot.  Consider a
  split union $\Wh(L)\sqcup R$ in $S^3$, and choose $m$ disjoint bands
  which join each component of $\Wh(L)$ to~$R$, such that in addition
  the bands are disjoint from an immersed ribbon disc for $R$ in $S^3$
  and are disjoint from Seifert surfaces for $\Wh(L)$.  Then the knot
  $K$ obtained from $\Wh(L)\sqcup R$ by these band sum operations
  along the arcs is slice.  }

\begin{proof}[Proof of
  Corollary~\ref{corollary:main-whitehead-ribbon-slice}]

  We first observe that a ribbon knot can be viewed as the result of
  band sum operations performed on a trivial link. Given a ribbon
  immersion $D^2 \looparrowright S^3$, by removing an
  $\epsilon$-neighbourhood of the singularities meeting the boundary
  of $D^2$, we obtain disjoint embedded discs, which are bounded by a
  trivial link.  This is indeed undoing band sum operations, since
  each removed $\epsilon$-neighborhood can be replaced as a band.

  Now, choose a ribbon embedding bounded by the given ribbon knot $R$.
  We may assume that the feet of the bands used to produce $K$ from
  $\Wh(L)\sqcup R$ are disjoint from the ribbon singularities.  Undo
  the band sum operations, to obtain $\Wh(L)\sqcup R$ from $K$, and
  then undo the band sum operations for the ribbon knot $R$ as in the
  previous paragraph, to transform $\Wh(L)\sqcup R$ into a split union
  of $\Wh(L)$ and a trivial link.  A trivial link is the Whitehead
  double of a trivial link, say $L_0$, from which it follows that our
  knot $K$ is obtained by band sum operations from the Whitehead
  double of the split union $L \sqcup L_0$.  Since both $L$ and $L_0$
  are homotopically trivial, $L\sqcup L_0$ is homotopically trivial.
  By Theorem~\ref{theorem:main-slice-knots}, $K$ is a slice knot.
\end{proof}

\subsection{Attempts to apply a surgery method}
\label{subsection:properties-of-slice-knots}

A standard surgery theoretic slicing process for a given knot $K$ with
zero-surgery manifold $M_K$ starts with an epimorphism $\pi_1(M_K)\to
G$ onto an appropriate ribbon group~$G$; here a ribbon group is the
fundamental group of the complement of a slicing disc in $D^4$
obtained by resolving singularities of a ribbon immersion
$D^2\looparrowright S^3$.  Then one applies topological surgery over
the group $G$, to obtain a slice disc exterior whose fundamental group
is~$G$. (For implementations of this strategy for knots, see for example~\cite{Freedman-Quinn:1990-1, Friedl-Teichner:2005-1}, while for links see~\cite{Cochran-Friedl-Teichner:2006-1}.)  With
current knowledge the surgery strategy can only be completed for
certain special cases, since it is unknown whether surgery works for
all ribbon groups.  There are only two ribbon groups for which surgery
is known to work: $\Z$ and $G_{6_1}:=\langle a, t\mid
ta^2t^{-1}=a\rangle$, a ribbon group for the Stevedore knot~$6_1$.
They were used in~\cite{Freedman-Quinn:1990-1, Friedl-Teichner:2005-1}
respectively.  We will show why this surgery approach fails to find slice disc exteriors for our knots from Corollary~\ref{corollary:main-whitehead-ribbon-slice} in many cases.


\begin{proposition}
  \label{proposition:properties-of-slice-knots}
  Suppose $K$ is a slice knot obtained by band sum operations on
  $\Wh(L)\sqcup R$ as in
  Corollary~\ref{corollary:main-whitehead-ribbon-slice}.
  \begin{enumerate}
 \item\label{item:slice-knot-alexander-poly} The knots $K$ and $R$
    have $S$-equivalent Seifert matrices.  Consequently, $\Delta_K(t)
    = \Delta_R(t)$.
 \item \label{item:slice-knot-epimorphism} There is
    an epimorphism of $\pi_1(M_K)$ onto $\pi_1(M_R)$ which takes a
    meridian to a meridian.
 \end{enumerate}
\end{proposition}

Before proving
Proposition~\ref{proposition:properties-of-slice-knots}, we discuss
some of its consequences.  First, from
Proposition~\ref{proposition:properties-of-slice-knots}~(\ref{item:slice-knot-epimorphism}),
it follows that for any ribbon group $G$ for $R$, there is an
epimorphism of $\pi_1(S^3 \sm K)$ onto~$G$, as mentioned in the
introduction.  The second consequence is that if $\Delta_R(t)$ is not
one and is not divisible by $t-\frac12$ (over~$\Q$), then the
topological surgery method to construct a slice disc exterior, which
we discussed in the introduction, does not work for~$K$.  In fact, to
apply topological surgery, one needs to start with an epimorphism of
$\pi_1(S^3\sm K)$ onto a ribbon group $G$ for which surgery is known
to work; the only such ribbon groups are $\Z$ and $G_{6_1}:=\langle a,
t\mid ta^2t^{-1}=a\rangle$, a ribbon group for the Stevedore knot
$6_1$.  For the $G=\Z$ case, it is known that the surgery programme
slices a knot if and only if $\Delta_K(t)=1$, essentially because
defining a surgery problem requires a degree one normal map with
target $S^1 \times D^3$ which restricts to a $\Z[\Z]$ homology
equivalence on the boundary.  By
Proposition~\ref{proposition:properties-of-slice-knots}~(\ref{item:slice-knot-alexander-poly}),
it follows that surgery cannot be carried out if $\Delta_R(t)\ne 1$.
Also, for $G=G_{6_1}$, if there were an epimorphism $\pi_1(S^3\sm K)
\to G_{6_1}$, then it would imply that $\Delta_K(t)$ is divisible by
the ``Alexander polynomial'' $\Delta_{G_{6_1}}(t)$ of $G_{6_1}$, which
is defined to be the order of the module
$(G_{6_1}'/G_{6_1}'')\otimes\Q \cong H_1(G_{6_1}';\Q)$ over the PID
$\Q[t^{\pm1}]$ as usual.  Indeed, $G_{6_1}$ is isomorphic to the
Baumslag-Solitar group $\big(\Z[t^{\pm1}] / (t-\frac12)\big) \rtimes
\langle t \rangle$, and we have $\Delta_{G_{6_1}}(t)=t-\frac12$.  From
this it follows that if $\Delta_R(t)$ is not divisible by $t-\frac12$,
then there is no epimorphism of $\pi_1(S^3\sm K)$ onto~$G_{6_1}$.  It
is conceivable that $K$ is smoothly concordant to a knot $J$ with
$\Delta_J = \Delta_{6_1}$, such that $J$ can be sliced using
\cite{Friedl-Teichner:2005-1}.  In this eventuality the resulting
slice disc would \emph{not} be homotopy ribbon.

Next we prove Proposition~\ref{proposition:properties-of-slice-knots}.

\begin{proof}[Proof of
  Proposition~\ref{proposition:properties-of-slice-knots}]

  First we prove part~(\ref{item:slice-knot-alexander-poly}).  The
  standard genus 1 Seifert surface of each component of $\Wh(L)$ has
  Seifert matrix
  $\left[\begin{smallmatrix}0&1\\0&1\end{smallmatrix}\right]$.  Since
  $L$ is homotopically trivial, each pairwise linking number of $L$ is
  zero.  It follows that the diagonal block sum of $m$ copies of
  $\left[\begin{smallmatrix}0&1\\0&1\end{smallmatrix}\right]$ and a
  Seifert matrix of $R$ is a Seifert matrix of $K$.  Consequently, $K$
  has a Seifert matrix $S$-equivalent to that of~$R$.

  The proof of part~(\ref{item:slice-knot-epimorphism}) follows
  immediately from Lemma~\ref{lemma:epimorphism-for-band-sum} below.
\end{proof}

\begin{lemma}
  \label{lemma:epimorphism-for-band-sum}
  Suppose $L$ is an $m$-component boundary link with Seifert surface
  $V$, and $J$ is a knot.  Suppose $K$ is a knot obtained from the
  split union $L\sqcup J$ by $m$ band sum operations along bands which
  join a component of $L$ to $J$ and whose interior is disjoint
  from~$V$.  Then there is an epimorphism $\pi_1(M_K)\to \pi_1(M_J)$
  which takes a meridian to a meridian.
\end{lemma}

\begin{proof}
  Let $\gamma_i$ be the core arc of the band joining $J$ and the $i$th
  component $L_i$ of~$L$, for $i=1,\ldots,m$.  Using Kirby calculus it
  is not too hard to see (see e.g.\
  \cite[Proof~of~Theorem~4.1]{Cochran-Orr-Teichner:2002-1}) that the
  3-manifold $M_K$ is obtained from $M_{L\sqcup J}$ by zero-framed
  surgery along $m$ curves, say~$\alpha_i$, each of which bounds an
  embedded 2-disc that meets $L\sqcup J$ at two transverse
  intersection points, contains $\gamma_i$, and induces the framing
  of~$\gamma_i$.  See Figure~\ref{figure:band-sum-surgery}.

  \begin{figure}[H]
    \labellist\small
    \pinlabel{$\gamma_i$} at 12 28
    \pinlabel{$\alpha_i$} at 75 40
    \pinlabel{$L_i$} at 84 63
    \pinlabel{$J$} at 97 12
    \endlabellist
    \includegraphics{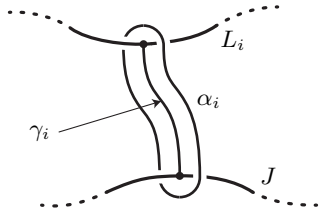}
    \caption{Band sum and surgery.}
    \label{figure:band-sum-surgery}
  \end{figure}

  The standard Pontryagin-Thom construction applied to the
  $m$-component Seifert surface $V$ gives a map
  \[
  S^3\sm \nu(L) \to \bigvee^m S^1
  \]
  which takes $S^3\sm \nu(V)$ to the wedge point; this induces an
  epimorphism $\phi\colon \pi_1(M_L) \to F :=$ free group of rank $m$,
  which takes a meridian of $L_i$ to the $i$th generator $x_i\in F$.

  Note that since $M_{L \sqcup J} \cong M_L \# M_J$, we have that
  $\pi_1(M_{L\sqcup J}) = \pi_1(M_{L})*\pi_1(M_J)$.  Furthermore, from
  the hypothesis that the interior of the arc $\gamma_i$ is disjoint
  from the Seifert surface $V$, it follows that
  $(\phi*\id)(\alpha_i) \in F * \pi_1(M_J)$ is of the form
  $x_i \zeta_i$ where $\zeta_i \in \pi_1(M_J)$.  Here $\phi*\id \colon \pi_1(M_{L})*\pi_1(M_J) \to F * \pi_1(M_J)$ is the map defined by sending elements of $\pi_1(M_J)$ in the free product to their image under~$\phi$.

  Let $W$ be the cobordism between $M_K$ and $M_{L\sqcup J}$
  obtained by attaching $m$ 2-handles along the curves $\alpha_i$ to
  the product $M_{L\sqcup J}\times[0,1]$.  We have an epimorphism
  \begin{equation}\label{equation:epimorphism-of-cobordism}
    \pi_1(W) \cong \frac{\pi_1(M_{L\sqcup
        J})}{\langle\alpha_1,\ldots,\alpha_m\rangle} =
    \frac{\pi_1(M_{L})*\pi_1(M_J)}{\langle\alpha_1,\ldots,\alpha_m\rangle}
    \xrightarrow{\phi*\id} \frac{F * \pi_1(M_J)}{\langle
      x_1\cdot \zeta_1, \ldots, x_m\cdot\zeta_m\rangle} \cong \pi_1(M_J).
  \end{equation}
  Also, turning $W$ upside down, $W$ is obtained by attaching
  2-handles to $M_K\times[0,1]$.  It follows that the inclusion
  induces an epimorphism $\pi_1(M_K) \to \pi_1(W)$.  Composing this
  with \eqref{equation:epimorphism-of-cobordism}, we obtain the desired
  epimorphism $\pi_1(M_K) \to \pi_1(M_J)$.  By construction, this
  takes a meridian to a meridian.
\end{proof}

Friedl and Teichner proposed conjectures related to necessary and
sufficient conditions for being homotopy ribbon, in \cite[Conjectures
1.6 and~1.8]{Friedl-Teichner:2005-1}.  It is an interesting question
whether the slice knots produced by
Corollary~\ref{corollary:main-whitehead-ribbon-slice} satisfy their
proposed conditions, that is, the Ext condition in
\cite[Conjecture~1.6]{Friedl-Teichner:2005-1} and the Poincar\'e
duality condition in \cite[Conjecture~1.8]{Friedl-Teichner:2005-1}.

\section{Slice links}
\label{section:slice-links}

In this section we present new slice links in $S^3$, using
Theorem~\ref{theorem:main} on distorted Casson towers.  We focus on a
constructions of links involving iterated Whitehead doubling, which is
naturally related to Casson towers.

\subsection{Kirby diagrams for distorted Casson towers}
\label{subsection:kirby-diagram-for-distorted-casson-tower}

In this subsection we discuss Kirby diagrams for arbitrary distorted
Casson towers, prior to their application in our construction of
another family of new slice links, presented in
subsection~\ref{subsection:distorted-iterated-ramified-whitehead-double}.

First we recall that a plumbing operation between two 2-handles in a
Kirby diagram gives us a new 1-handle and a clasp (whose sign is equal
to the sign of the plumbing) between the attaching circles of the
2-handles, as shown in Figure~\ref{figure:plumbing-kirby-diagram}.  As
a reference, see for
instance~\cite[Example~6.1.3]{Gompf-stipsicz-book}.

\begin{figure}[H]
  \labellist
  \small
  \pinlabel{$h$} at 10 5
  \pinlabel{$h'$} at 52 5
  \pinlabel{$h$} at 140 5
  \pinlabel{$h'$} at 182 5
  \endlabellist
  \includegraphics{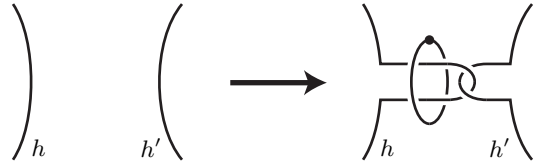}
  \caption{Plumbing of 2-handles in a Kirby diagram.}
  \label{figure:plumbing-kirby-diagram}
\end{figure}

Start with a standard Kirby diagram of a Casson tower. For example,
see Figure~\ref{figure:casson-tower-diagram}, which is a Casson tower
of height~4.  (It is a good exercise, for those not familiar with this
diagram, to build it using the above plumbing operation; or see
\cite{Casson-1986-towers}, \cite[Section~2]{Freedman:1982-1},
\cite[Chapter~12]{Freedman-Quinn:1990-1},
\cite[Example~6.1.3]{Gompf-stipsicz-book}.)  So far all the plumbing
operations are self-plumbings.  As our temporary convention, a circle
without a dot or a label designates a zero-framed attaching circle of
a 2-handle.

\begin{figure}[H]
  \labellist
  \footnotesize
  \pinlabel{$C(T)$} at -5 123
  \endlabellist
  \includegraphics[scale=.95]{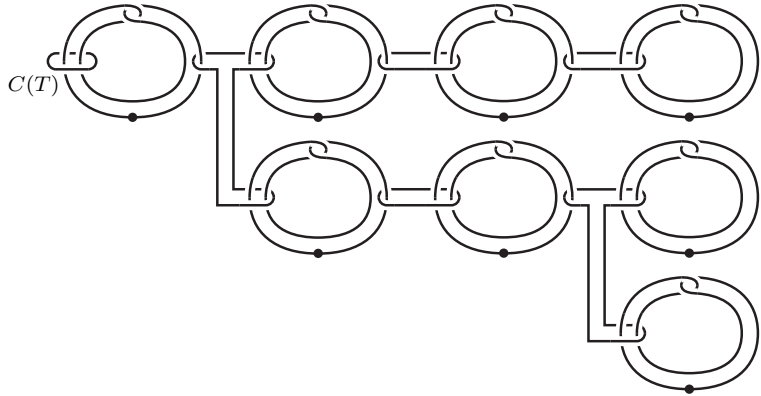}
  \caption{A Casson tower of height~4.  Unlabelled circles without a
    dot are zero-framed.}
  \label{figure:casson-tower-diagram}
\end{figure}

By applying the above plumbing operation for 2-handles, we can plumb a
stage 4 disc to another disc of stage two or higher.  For instance, by
plumbing a stage 4 disc in Figure~\ref{figure:casson-tower-diagram} to
a stage 2 disc and by plumbing another stage 4 disc to a stage 3 disc,
we obtain a distorted Casson tower described by the Kirby diagram in
Figure~\ref{figure:distorted-casson-tower-diagram}.

\begin{figure}[H]
  \labellist
  \footnotesize
  \pinlabel{$C(T)$} at -5 146
  \endlabellist
  \includegraphics[scale=.95]{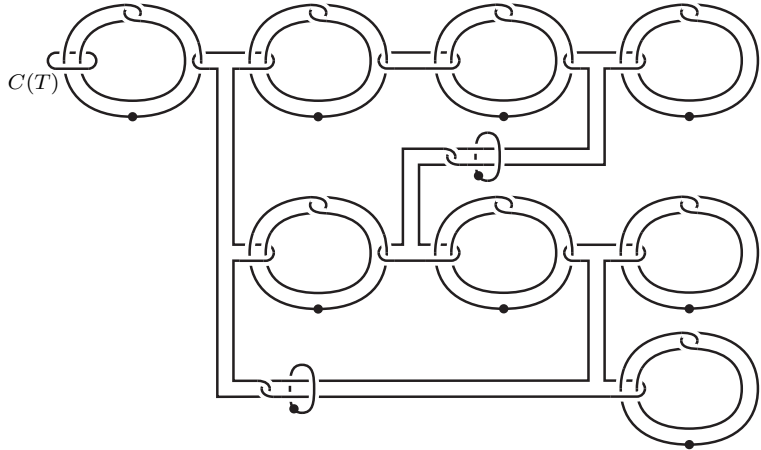}
  \caption{A Kirby diagram of a distorted Casson tower of height~4.}
  \label{figure:distorted-casson-tower-diagram}
\end{figure}

\subsubsection*{Elimination of all 2-handles}

In case of a (non-distorted) Casson tower, it is well known that one
can eliminate 1- and 2-handles in pairs to obtain a Kirby diagram
without 2-handles.  A standard procedure, which is a ``top-to-bottom''
elimination, is as follows.  Start with a diagram such as
Figure~\ref{figure:casson-tower-diagram} (for which this procedure
will be ``right-to-left'').  Slide each rightmost 1-handle, which is
associated to a self-intersection of the top stage disc, under the
adjacent 1-handle on its left.  Then eliminate the 1-handle under
which it slid, together with the linking 2-handle.  Iterate this, to
obtain a Kirby diagram with $k$ 1-handles, where $k$ is the number of
self-intersections of the final stage discs.  In fact $C(T)$ remains
as an unknotted circle, and the 1-handles form a ramified iterated
Whitehead double of a meridian of~$C(T)$.

For a distorted Casson tower diagram, the plumbings performed between
2-handles as in Figure~\ref{figure:plumbing-kirby-diagram} may prevent
the elimination of a 1-handle and a linking 2-handle in the above
procedure.  Instead, in order to obtain a Kirby diagram of a distorted
Casson tower without 2-handles, we will perform the ``bottom-to-top''
elimination discussed below, which is ``left-to-right'' in case
of Figure~\ref{figure:distorted-casson-tower-diagram}.

For this elimination the (well known) modification shown in
Figure~\ref{figure:whitehead-slide-move} is useful; $R$ denotes an
arbitrary tangle diagram, and the hatched band designates parallel
strands.  The box with label $-2$ designates two negative full twists.
(If we had a negative clasp in the first picture, we would have $+2$
instead.)  The first step in Figure~\ref{figure:whitehead-slide-move}
is an isotopy which straightens the 1-handle, and the second step is
handle sliding and cancellation.

\begin{figure}[H]
  \labellist
  \footnotesize
  \pinlabel{$0$} at 100 119
  \pinlabel{$0$} at 255 118
  \pinlabel{$-2$} at 188 49
  \pinlabel{$-2$} at 188 118
  \normalsize
  \pinlabel{$R$} at 14 96
  \pinlabel{$R$} at 157 96
  \pinlabel{$R$} at 157 27
  \endlabellist
  \includegraphics[scale=.95]{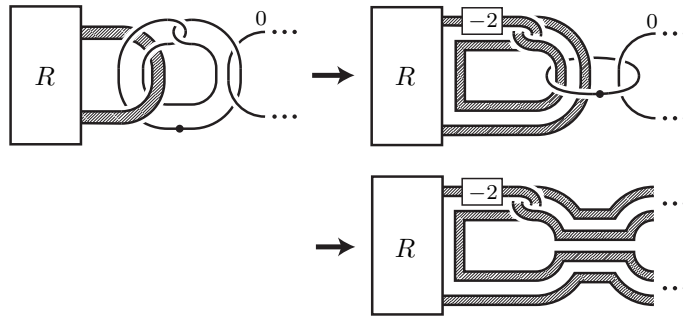}
  \caption{A modification of a Kirby diagram.}
  \label{figure:whitehead-slide-move}
\end{figure}

Now start with a Kirby diagram drawn as in
Figure~\ref{figure:distorted-casson-tower-diagram}.  First apply the
move in Figure~\ref{figure:whitehead-slide-move} to the leftmost part
to eliminate the leftmost 1-handle, and a 2-handle linking it.  By
this $C(T)$ becomes the Whitehead double of the attaching circle of
the eliminated 2-handle; see
Figure~\ref{figure:elimination-in-casson-tower-diagram}.  Repeatedly
apply the move in Figure~\ref{figure:whitehead-slide-move} to
eliminate the next leftmost 1-handles and 2-handles in pairs.
Eventually we obtain a Kirby diagram with $k$ 1-handles and no
2-handles, where $k$ is the number of intersections of the top stage
discs and stage $\ge 2$ discs.  For instance, if we start with
Figure~\ref{figure:distorted-casson-tower-diagram}, we have $k=5$
since there are 3 self-intersections of stage 4 discs, and 2
``distorting'' intersections between stage 4 and lower stage discs.
See Figures~\ref{figure:elimination-in-casson-tower-diagram-2}
and~\ref{figure:elimination-in-casson-tower-diagram-3}.

A consequence of this is that a distorted Casson handle is
diffeomorphic to the boundary connected sum of $k$ copies of
$S^1\times D^3$.

\begin{figure}[H]
  \labellist
  \footnotesize
  \pinlabel{$C(T)$} at -2 78
  \endlabellist
  \includegraphics[scale=.95]{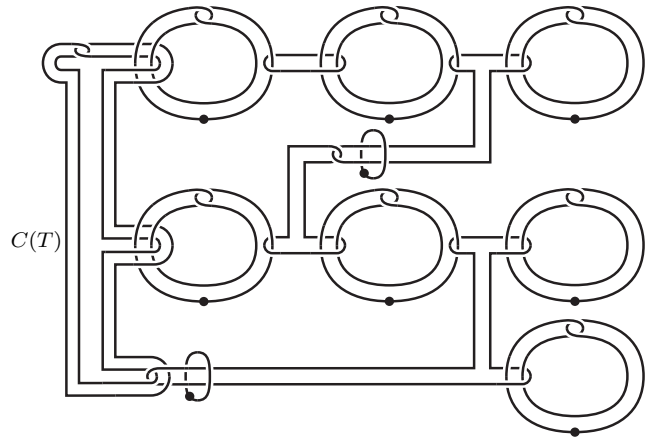}
  \caption{A distorted Casson tower diagram with a 1-handle and a
    2-handle eliminated.}
  \label{figure:elimination-in-casson-tower-diagram}
\end{figure}

\begin{figure}[H]
  \labellist
  \footnotesize
  \pinlabel{$C(T)$} at 5 9
  \pinlabel{$-2$} at 36 206
  \pinlabel{$-2$} at 35 152
 \endlabellist
  \includegraphics{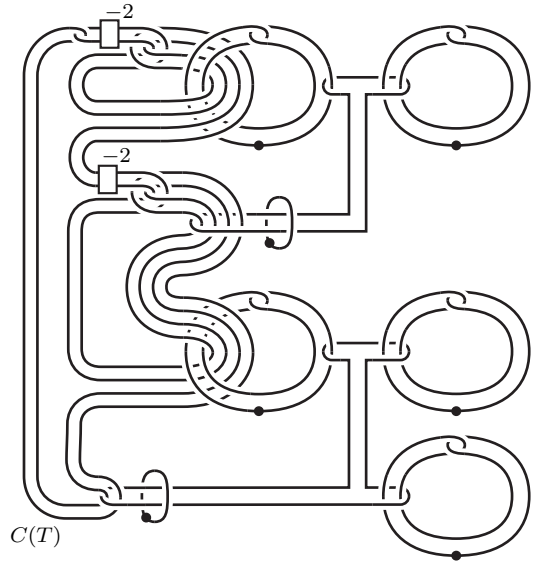}
  \caption{Further elimination of handles in a distorted Casson tower
    diagram.}
  \label{figure:elimination-in-casson-tower-diagram-2}
\end{figure}

\begin{figure}[H]
  \labellist
  \footnotesize
  \pinlabel{$C(T)$} at 5 46
  \pinlabel{$-2$} at 43 285
  \pinlabel{$-2$} at 131 285
  \pinlabel{$-2$} at 35 166
  \pinlabel{$-2$} at 131 166
  \endlabellist
  \includegraphics[scale=.9]{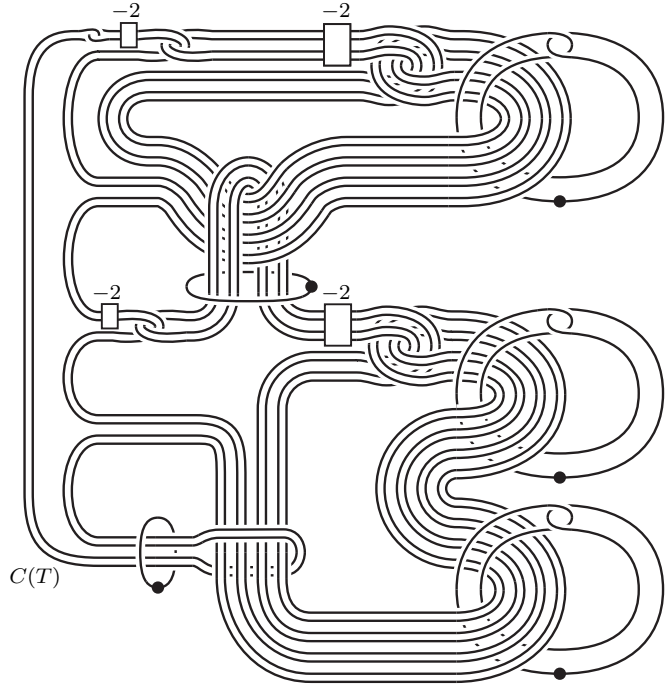}
  \caption{A distorted Casson tower diagram with all 2-handles
    eliminated.}
  \label{figure:elimination-in-casson-tower-diagram-3}
\end{figure}

\subsection{Distorted 4-fold iterated ramified Whitehead doubles}
\label{subsection:distorted-iterated-ramified-whitehead-double}

In this subsection we give a new family of slice links.  The main
ingredients are Theorem~\ref{theorem:main} on distorted Casson towers
of height~4 and the Kirby diagrams we obtained in
Section~\ref{subsection:kirby-diagram-for-distorted-casson-tower}.
First we begin with a general construction of links, without requiring
a distorted Casson tower.  Then we will relate such a link to a
distorted Casson tower, by connecting combinatorial choices involved
in the construction of the link to intersection data of the
corresponding distorted Casson tower.

\subsubsection*{Construction of links}

We start with the split union of arbitrary number of Hopf
links.  Choose one component from each Hopf link, and denote the union
of the chosen components by~$L_1$.  Denote the union of the other
components by~$L_2$.  In what follows,
Whitehead doubles and parallels are always untwisted, and taken in a
tubular neighbourhood which is thin enough to be disjoint from anything
we have considered previously.  Also, a band sum is always assumed to
be between components of split sublinks along a ``straight'' band;
more precisely, whenever two components $J$ and $J'$ of a link are
joined by a band, there is a separating 2-sphere $S$ in $S^3$ disjoint
from the link, and the band passes through $S$ exactly once and is
disjoint from anything we have considered previously.  This determines
the result of the band sum uniquely up to isotopy.  Now the
construction is described below.

\begin{enumerate}
\item\label{step1-link-construction} Replace each component of $L_2$
  with~$\Wh(L_2)$.  Perform some band sum operations to combine
  distinct components of~$L_1$ and call the result $L_1'$.  Remember a
  meridian of each component of $L_1'$ for later use, without adding
  it to the link.
\item\label{step2-link-construction} Replace $L_1'$ with $\Wh(L_1')$,
  perform some band sum operations to combine distinct components of
  $\Wh(L_1')$, and call the result~$L_1''$.  The sublink $\Wh(L_2)$ is
  left unchanged.  Remember a meridian of~$L_1''$.
\item\label{step3-link-construction} Perform (2) once again for
  $L_1''$ in place of $L_1'$ and call the result~$L_1'''$.  Remember a
  meridian of each component of $L_1'''$ for later use.
\item\label{step4-link-construction} Perform (2) once again for
  $L_1'''$ in place of~$L_1'$.  This time we perform band sum
  operations on $\Wh(L_1''')$ until we obtain a knot, say~$J$.
\item\label{step5-link-construction} Perform the following operation
  some number of times: choose a remembered meridian of a component of
  $L_1'$ and a remembered meridian of a component of either $L_1'$,
  $L_1''$ or~$L_1'''$.  Band sum them, add a meridional circle of the
  band to our link, and modify $J$ by performing $\pm1$ surgery on the
  banded together meridians, then $\mp 1$ surgery on each of the
  meridians individually.  This introduces a clasp between strands
  enclosed by the meridians.  If the same meridian is chosen more than
  once during the iteration, use a parallel copy.
\end{enumerate}

The final outcome is the union of $\Wh(L_2)$ and $J$ modified in
Step~$($\ref{step5-link-construction}$)$.  Remembered meridians are
not included.

Maybe our construction is best understood by an example: see
Figure~\ref{figure:slice-link-construction}.

\begin{figure}[H]
  \labellist
  \small
  \pinlabel{$L_1$} at 5 378
  \pinlabel{$L_2$} at 48 378
  \pinlabel{(1)} at 66 354
  \pinlabel{$L_1'$} at 98 378
  \pinlabel{$\Wh(L_2)$} at 147 408
  \pinlabel{(2)} at 163 354
  \pinlabel{$L_1''$} at 194 378
  \pinlabel{$\Wh(L_2)$} at 256 412
  \pinlabel{(3)} at 277 354
  \pinlabel{$L_1'''$} at 308 390
  \pinlabel{$\Wh(L_2)$} at 405 389
  \pinlabel{(4)} at 245 268
  \pinlabel{$J$} at 10 220
  \pinlabel{$\Wh(L_2)$} at 170 184
  \pinlabel{$\mu_1$} at 38 197
  \pinlabel{$\mu_2$} at 38 175
  \pinlabel{$\mu_3$} at 97 16
  \pinlabel{$\mu_4$} at 40 20
  \pinlabel{(5)} at 208 131
  \endlabellist
  \includegraphics[scale=.85]{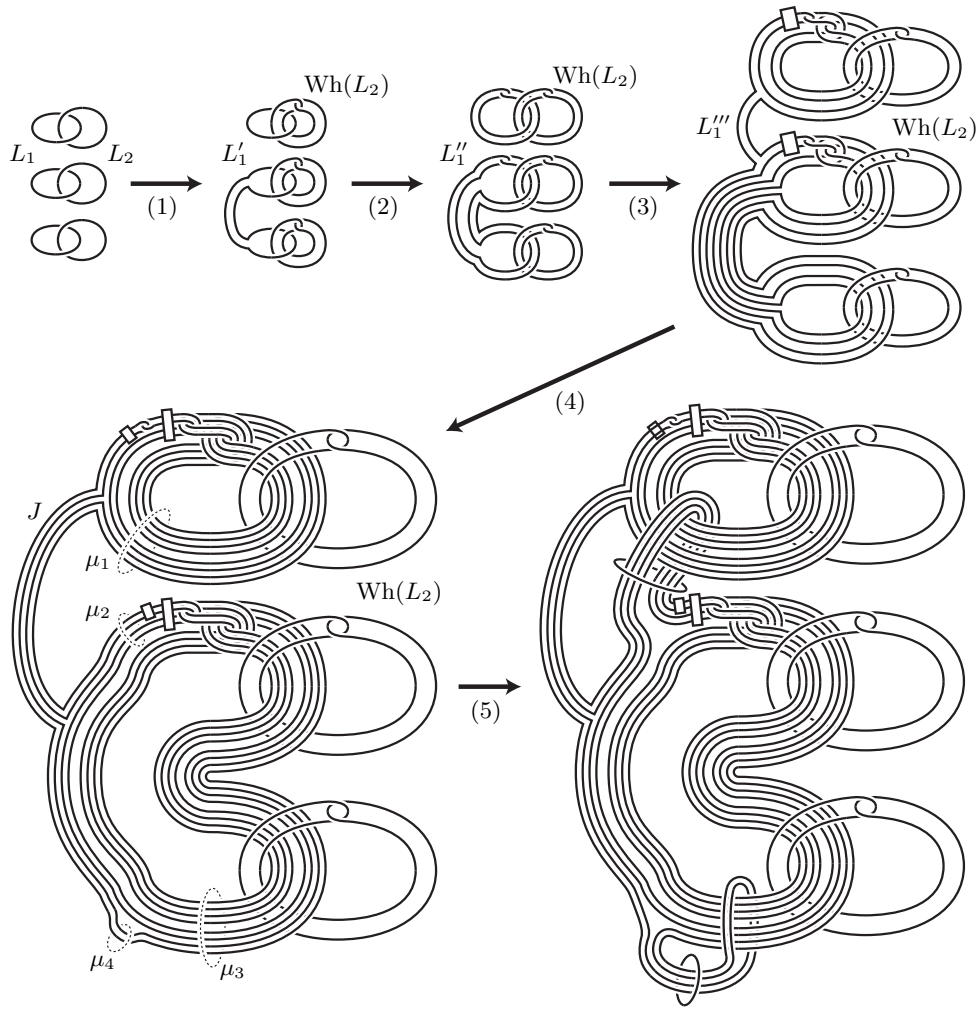}
  \caption{A construction of a slice link.  Each box designates $-2$
    full twists.  The meridians $\mu_1$, $\mu_2$, $\mu_3$, and $\mu_4$
    are those of $L_1'$, $L_1''$, $L_1'$, and $L_1'''$, respectively,
    and Step (5) is performed for the pairs $(\mu_1, \mu_2)$ and
    $(\mu_3, \mu_4)$ to obtain the last link as the final outcome.}
  \label{figure:slice-link-construction}
\end{figure}

\begin{theorem}
  \label{theorem:distorted-casson-tower-slicing}
  Any link constructed as above is slice.
\end{theorem}

By Theorem~\ref{theorem:distorted-casson-tower-slicing}, the last link
in Figure~\ref{figure:slice-link-construction} is slice.  As another
example which is simpler, the link in
Figure~\ref{figure:distorted-tower-slice-link} is slice.  Thus
Theorem~\ref{theorem:main-distorted-casson-tower-link-example} is a
consequence of Theorem~\ref{theorem:distorted-casson-tower-slicing}.
Indeed, the link in Figure~\ref{figure:distorted-tower-slice-link} is
obtained by applying the above construction to a distorted Casson
tower of height 4 which has one plumbed handle with one self-plumbing
at each stage and has one distorting intersection between the stage 4
and stage 2 discs.

\begin{proof}[Proof of Theorem~\ref{theorem:distorted-casson-tower-slicing}]

  We claim that a link $L$ obtained by the above construction is the
  union of the curve $C(T)$ and the dotted circles representing
  1-handles in a Kirby diagram of a height 4 distorted Casson tower
  $T$ without 2-handles.  For instance, observe that
  Figure~\ref{figure:elimination-in-casson-tower-diagram-3} and the
  final picture in Figure~\ref{figure:slice-link-construction} are
  identical.  In fact, our construction of $L$ corresponds to a
  top-to-bottom construction of a distorted Casson tower, as follows.
  For each component of $\Wh(L_2)$ in
  Step~(\ref{step1-link-construction}), take a disc with a single
  local self-plumbing, which we call a pre-stage-4 disc.  Whenever we
  perform band sum of components of $L_1$ in Step
  (\ref{step1-link-construction}), take a boundary connected sum of
  the associated pre-stage-4 discs.  The resulting discs with
  (multi-)self-plumbings are the stage 4 discs of our tower $T$.  In
  Step (\ref{step2-link-construction}), whenever we take a Whitehead
  double of a component, take a disc with a single local
  self-plumbing, which we call a pre-stage-3 disc, and attach the
  associated stage 4 disc to the pre-stage-3 disc along the double
  point loop.  Again, whenever we perform a band sum, take the
  boundary connected sum of the pre-stage-3 discs.  The result is
  stage 3 discs with the stage 4 discs attached.  Continue in the same
  way for steps (\ref{step3-link-construction}) and
  (\ref{step4-link-construction}) to produce stages 2 and~1.  We
  arrive at a non-distorted Casson tower of height~4.  Finally, for
  each triple of $\pm1$ surgeries occurring in
  step~(\ref{step5-link-construction}), plumb a stage 4 disc to a
  stage 4, 3 or 2 disc, where the choice of the meridians determines
  which discs to plumb.  The 2-handle elimination procedure described
  in Section~\ref{subsection:kirby-diagram-for-distorted-casson-tower}
  applies to the standard Kirby diagram of the resulting distorted
  Casson tower, from which the claim follows.  The lemma stated below
  now completes the proof.
\end{proof}

\begin{lemma}
  \label{lemma:distorted-casson-tower-slicing}
  Let $L$ be the union of the curve $C(T)$ and the dotted circles
  representing 1-handles in a Kirby diagram of a height 4 distorted
  Casson tower $T$ without 2-handles.  Then, as a link in $S^3$, $L$
  is slice.
\end{lemma}

\begin{proof}
  The Kirby diagram without 2-handles determines an embedding of $T$
  into the 4-ball, to wit, $T$ is the exterior of the standard slicing
  discs $\Delta_i$ bounded by the dotted circles (which form a trivial
  link).  By Theorem~\ref{theorem:main}, the curve $C(T)$ bounds a
  flat disc $\Delta$ in~$T$.  As $\Delta$ is disjoint from the discs
  $\Delta_i$, $L$ is slice.
\end{proof}

The above lemma also applies to give another family of slice links:
recall that Theorem~\ref{theorem:main-ramified-wh4} in the
introduction states that \emph{any ramified $\Wh_n$ link is slice for
  $n\ge 4$}.

\begin{proof}[Proof of Theorem~\ref{theorem:main-ramified-wh4}]
  It suffices to show that any ramified $\Wh_4$ link is slice.  Recall
  that a ramified $\Wh_4$ link $L$ is the union of $C(T)$, and the
  dotted circles in a Kirby diagram obtained by a top-to-bottom
  elimination of 2-handles applied to the standard Kirby diagram of a
  (non-distorted) Casson tower of height 4; here $C(T)$ remains as an
  unknotted circle and the other components form a 4th iterated
  ramified Whitehead double of a meridian of~$C(T)$.  By
  Lemma~\ref{lemma:distorted-casson-tower-slicing}, $L$ is slice.
\end{proof}

\section{The grope filtration of knots and Casson towers of height~3}
\label{section:casson-tower-3-infinite-grope}

In this section we make the observation that we can use the improved
initial hypothesis in the Grope Height Raising Lemma to slightly
extend results from~\cite{Ray-2013-1} on the grope filtration of the
knot concordance group. The grope filtration first appeared in the
literature in~\cite{Cochran-Teichner:2003-1}, although it was already
implicit in~\cite{Cochran-Orr-Teichner:1999-1}.  By definition a knot
in $S^3$ lies in the $n$th term $\mathcal{G}^{(n)}$ of the filtration
if it bounds an embedded framed grope $G_n$ in $D^4$.  Ray shows that
a knot in $S^3$ which bounds a Casson tower of height~3 is
$(n)$-solvable for all $n$.  This follows from the corollary below,
by~\cite[Theorem~8.11]{Cochran-Orr-Teichner:1999-1}.  She also shows
that a knot which bounds a Casson tower of height $n$ bounds a grope
of height $n$~\cite[Theorem~A~(i)]{Ray-2013-1}.  In fact, a height~3
Casson tower is enough to obtain this conclusion for all $n$.

\begin{corollary}\label{cor:grope-filtration}
  A Casson tower $T_3$ of height 3 contains an embedded framed grope
  $G_n$ of height $n$, with the same attaching circle as the Casson
  tower, for all $n$.
\end{corollary}

\begin{proof}
  Apply Proposition~\ref{proposition:grope-in-casson-tower} to
  construct a properly immersed framed capped grope of height~2 inside
  $T_3$, as described in the beginning of the proof of
  Theorem~\ref{theorem:main}\@.  Apply grope height raising, as in
  Lemma~\ref{lemma:grope-height-raising}, to obtain a properly
  immersed framed capped grope of height~$n$, and then ignore the
  caps.
\end{proof}

It is interesting to contrast Corollary~\ref{cor:grope-filtration}
with Theorem~\ref{theorem:main-height-3}.

As mentioned above, Ray showed in \cite{Ray-2013-1} that a link in
$S^3$ which bounds a Casson tower of height~3 in~$D^4$ is
$(n)$-solvable, in the sense of \cite{Cochran-Orr-Teichner:1999-1},
for all~$n$.  Corollary~\ref{cor:grope-filtration} shows that the link
also lies in the intersection of the grope filtration.  A link which
bounds a grope of height $n+2$ is
$(n)$-solvable~\cite[Theorem~8.11]{Cochran-Orr-Teichner:1999-1}, but
by~\cite[Corollary~6.8]{Otto:2014-1} the converse does not hold for
links of at least $2^{n+2}$ components. So for finite terms of the
filtrations, having a height $n+2$ embedded grope is stronger than
having an $(n)$-solution.  As observed in \cite{Ray-2013-1}, we can
deduce the existence of an $(n)$-solution from the existence of an
immersed grope of height $n+2$ with the bottom two stages embedded.
It is not known whether the infinite intersections of the filtrations
coincide.

\bibliographystyle{amsalpha}
\def\MR#1{}
\bibliography{research}

\end{document}